\documentclass{amsart}
\pdfoutput=1

\usepackage{amsmath, amssymb, amsthm, mathtools, enumitem}
\usepackage{accents, xcolor}
\usepackage{cancel}
\usepackage{amssymb}
\usepackage{amsthm}
\usepackage{amsfonts}
\usepackage{amsmath}
\usepackage{pmboxdraw}
\usepackage{verbatim} % for comment
\usepackage{graphicx}
\usepackage{color}
\usepackage[colorlinks=true, citecolor=blue, filecolor=black, linkcolor=black, urlcolor=black]{hyperref}
\usepackage{cite}
\usepackage[normalem]{ulem}
\usepackage{subcaption}
\usepackage{todonotes}
\usepackage{kantlipsum}
\allowdisplaybreaks
\usepackage{bbm}
\usepackage{tikz}
\usepackage{pgfplots} % for plotting gaussian function
\usetikzlibrary{decorations, arrows.meta, patterns}

\newtheorem{theorem}{Theorem}[section]
\newtheorem{lemma}[theorem]{Lemma}

\theoremstyle{definition}

\theoremstyle{remark}

\newcommand{\erfc}{\operatorname{erfc}}

\newcommand{\gaussF}{{}_{2} {F}_1}
\newcommand{\gausF}{{}_{1} {F}_1}
\newcommand{\Airy}{\operatorname{Ai}}
\newcommand{\re}{\operatorname{Re}}

\newcommand{\im}{\operatorname{Im}}
\newcommand{\sgn}{\operatorname{sgn}}
\newcommand{\Pf}{{\textup{Pf}}}
\newcommand{\Tr}{\operatorname{Tr}}
\newcommand{\rhohat}{\mathbf{\hat{\rho}}}

\def\C{\mathbb{C}}

\def\E{\mathbb{E}}
\def\HH{\mathbb{H}}

\def\R{\mathbb{R}}

\newcommand{\vz}{\boldsymbol{z}}
\newcommand{\vlambda}{\boldsymbol{\lambda}}
\newcommand{\vmu}{\boldsymbol{\mu}}
\newcommand{\vzeta}{\boldsymbol{\zeta}}

\numberwithin{equation}{section}

\title[The probability of almost all eigenvalues being real for the elliptic GinOE]{The probability of almost all eigenvalues being real\\for the elliptic real Ginibre ensemble}

\author{Gernot Akemann}
\address{ Faculty of Physics, Bielefeld University, P.O. Box 100131, 33501 Bielefeld, Germany}
\email{akemann@physik.uni-bielefeld.de}

\author{Sung-Soo Byun}
\address{Department of Mathematical Sciences and Research Institute of Mathematics, Seoul National University, Seoul 151-747, Republic of Korea}
\email{sungsoobyun@snu.ac.kr}

\author{Yong-Woo Lee}
\thanks{Corresponding author: Yong-Woo Lee (\texttt{hellowoo@snu.ac.kr})}
\address{Department of Mathematical Sciences, Seoul National University, Seoul 151-747, Republic of Korea}
\email{hellowoo@snu.ac.kr}

\begin{document}

\begin{abstract}
We investigate real eigenvalues of real elliptic Ginibre matrices of size $n$, indexed by the parameter of asymmetry $\tau \in [0,1]$. In both the strongly and weakly non-Hermitian regimes, where $\tau \in [0,1)$ is fixed or $1-\tau=O(1/n)$, respectively, we derive the asymptotic expansion of the probability $p_{n,n-2l}$ that all but a finite number $2l$ of eigenvalues are real. In particular, we show that the expansion is of the form 
\begin{align*}
\log p_{n, n-2l} = \begin{cases}
a_1 n^2 +a_2 n + a_3 \log n +O(1) &\textup{at strong non-Hermiticity},
\\
b_1 n +b_2 \log n + b_3   +o(1) &\textup{at weak non-Hermiticity},
\end{cases}
\end{align*}
and we determine all coefficients explicitly. Furthermore, in the special case where $l=1$, we derive the full-order expansions. For the proofs, we employ distinct methods for the strongly and weakly non-Hermitian regimes. In the former case, we utilise potential-theoretic techniques to analyse the free energy of elliptic Ginibre matrices conditioned to have \(n-2l\) real eigenvalues, together with the strong Szegő limit theorems. In the latter case, we utilise the skew-orthogonal polynomial formalism and the asymptotic behaviour of the Hermite polynomials.

\end{abstract}

\maketitle 

\pgfmathdeclarefunction{gauss}{2}{%
    \pgfmathparse{1/(#2*sqrt(2*pi))*exp(-((x-#1)^2)/(2*#2^2))}%
}

\section{Introduction and Main results}

In the study of point processes, one of the most natural observables of interest is their counting or number statistics, which addresses fundamental questions about how many particles can be observed within a given domain.  
From a probabilistic perspective, the natural questions that arise can be explored step by step as follows:  
\begin{itemize}
    \item \textit{What is the typical number of particles?}
    This involves determining both the expectation and variance of the observable, which are central to the law of large numbers.  
    \smallskip
    \item \textit{What are the fluctuations around the typical number of particles?}
    This question provides a more refined statistical description and relates to the central limit theorem.  
    \smallskip
    \item \textit{What is the probability of rare events?} 
    Here, one typically investigates the probability of observing an extremely small or large number of particles compared to the typical number. This falls within the scope of large deviation theory.  
\end{itemize}
These types of questions have been extensively studied in the literature and serve as central topics in the field. 
 
In this work, we address a question in this direction for a point process arising from non-Hermitian random matrix theory. Specifically, we focus on asymmetric random matrices with real entries, with the real Ginibre matrix (denoted by GinOE) serving as a prototypical example \cite{BF25}.  
For asymmetric random matrices, an interesting domain of study in the number statistics is the real axis. In other words, we examine the number statistics of real eigenvalues. 
The presence of real eigenvalues with non-zero probability is particularly characteristic of the orthogonal symmetry class, distinguishing it from the unitary and symplectic symmetry classes (see e.g. \cite{Luh18}). This question has been addressed in various models including the GinOE \cite{EKS94, FN07} and its variants including the elliptic GinOE \cite{FN08, BKLL23}, truncated orthogonal ensemble \cite{FIK20, FK18, LMS22}, products of GinOE matrices \cite{AB23, Si17a, Fo14, FI16}, spherical Ginibre matrices \cite{FM12,EKS94}, asymmetric Wishart matrices \cite{BN25}, as well as matrices with non-Gaussian entries \cite{TV15}. 
We also refer the reader to \cite{Ch22,ACCL24,ABES23,FKP24,DLMS24,GLX23,ADM24} and the references therein for the counting statistics of various two-dimensional point processes.
In addition, we mention that the eigenvalue statistics of the elliptic GinOE finds applications in various areas such as ecosystem stability \cite{BFK21,FK16} and complexity of random energy landscapes \cite{Kiv24,Fyo16}. 

\medskip 

We first introduce our model of study, the elliptic GinOE. The GinOE is an \( n \times n \) matrix $G$ with independent and identically distributed standard normal entries. Its elliptic extension is then defined as  
\begin{equation} \label{def of eGinibre}
X := \frac{\sqrt{1+\tau}}{2}(G+G^T)+\frac{\sqrt{1-\tau}}{2}(G-G^T),
\end{equation}
where \( \tau \in [0,1) \) is the asymmetry parameter. This parameter governs the extent to which the matrix deviates from being symmetric, allowing the elliptic GinOE to interpolate between the GinOE (\( \tau = 0 \)) and the Gaussian Orthogonal Ensemble (GOE) in the limit as \( \tau \to 1 \). 

We investigate the asymptotic properties of the real eigenvalues of the elliptic GinOE as \( n \to \infty \). When examining the asymptotic behaviour of the elliptic GinOE, two distinct regimes of interest emerge:  
\begin{itemize}
    \item \textbf{Strong Non-Hermiticity}: This is the case where \( \tau \in [0,1) \) is fixed as \( n \to \infty \).  
    \smallskip 
    \item \textbf{Weak Non-Hermiticity}: This corresponds to the case where \( \tau \to 1 \) at a specific rate. In our study, we consider \( \tau = 1 - \alpha^2/n \), where \( \alpha \in [0, \infty) \) is a fixed parameter. In this regime, one can observe critical transitions between asymmetric and symmetric matrices. This regime was introduced in the papers \cite{FKS97,FKS97a,FKS98}. 
\end{itemize}
For the readers' convenience, we provide a brief summary of the development of the study of real eigenvalues of the elliptic GinOE, both at strong and weak non-Hermiticity. 

The first natural question concerns the law of large numbers. To address this, let \( \mathcal{N} \equiv \mathcal{N}_\tau \) denote the (random) number of real eigenvalues of the elliptic GinOE. In order to establish the law of large numbers for the random variable $\mathcal{N}$ (viewed as a sequence indexed by $N$), one needs to determine the asymptotic behaviour of its mean and variance. It was established in \cite{FN08, BKLL23} (cf. see \cite{EKS94} for an earlier study on the GinOE case) that  
\begin{equation} \label{asymp expected number}
    \mathbb{E} \mathcal{N}  = \begin{cases}
        \mathfrak{b}(\tau) \sqrt{n}+O(1), &\textup{for } \tau \in [0,1) \text{ fixed}, 
        \smallskip 
        \\
        \mathfrak{c}(\alpha) \, n +O(1), &\textup{for } \tau = 1 - \frac{\alpha^2}{n},
    \end{cases}
\end{equation}
where the expectation is taken with respect to the elliptic GinOE, and 
\begin{equation} \label{eq. coeff avg number}
\mathfrak{b}(\tau):= \sqrt{\frac{2}{\pi} \frac{1+\tau}{1-\tau} }, \qquad   \mathfrak{c}(\alpha) := e^{-\frac{\alpha^2}{2}} \Big( I_0\big(\tfrac{\alpha^2}{2}\big) + I_1\big(\tfrac{\alpha^2}{2}\big) \Big).  
\end{equation}
Here, $I_\nu$ is the modified Bessel function of the first kind given by  
\begin{equation} \label{I nu}
I_\nu(z):= \sum_{k=0}^\infty \frac{(z/2)^{2k+\nu}}{k!\,\Gamma(\nu+k+1)},
\end{equation}
see e.g. \cite[Chapter 10]{NIST}. We mention that the coefficients \eqref{asymp expected number} can be realised as normalisation constants in the limiting macroscopic densities \cite{Efe97,BKLL23,BL24}.

It is evident from \eqref{asymp expected number} that a non-trivial proportion of eigenvalues is expected to be real in weak non-Hermiticity, whereas only a subdominant \( O(\sqrt{n}) \) number of eigenvalues are real in strong non-Hermiticity. This observation aligns with the intuition that the number of real eigenvalues increases as the ensemble becomes more symmetric.  
 In addition to the mean asymptotic of $\mathcal {N}$, its variance has been studied in \cite{FN08, BKLL23} as well, which reads
\begin{equation} \label{asymp variance}
\lim_{n \to \infty} \frac{ \textup{Var}\, \mathcal{N} }{ \mathbb{E} \mathcal{N}  } =  \begin{cases}
        2-\sqrt{2}, &\textup{for } \tau \in [0,1) \text{ fixed}, 
        \smallskip 
        \\
        2- 2 \dfrac{ \mathfrak{c}(\sqrt{2}\alpha)}{ \mathfrak{c}(\alpha) }, &\textup{for } \tau = 1 - \frac{\alpha^2}{n}.
    \end{cases}
\end{equation}
Combining \eqref{asymp expected number} and \eqref{asymp variance}, one can immediately observe the law of large numbers for $\mathcal{N}$ in probability. 

The next question concerns the fluctuations of the number of real eigenvalues, which were studied relatively recently in \cite{BMS23, FS23a, Fo24}. In these works, the asymptotic normality of the fluctuations was established. To be more precise, it was shown that
\begin{equation}
\frac{ \mathcal{N}-\mathbb{E} \mathcal{N} }{  \sqrt{\textup{Var}\, \mathcal{N}}  }  \to \textup{N}(0,1), \qquad  \textup{as } n \to \infty , 
\end{equation}
where $\textup{N}(0,1)$ denotes the standard Gaussian distribution.
For earlier results on this topic in the case of the GinOE, we refer the reader to \cite{Si17}.

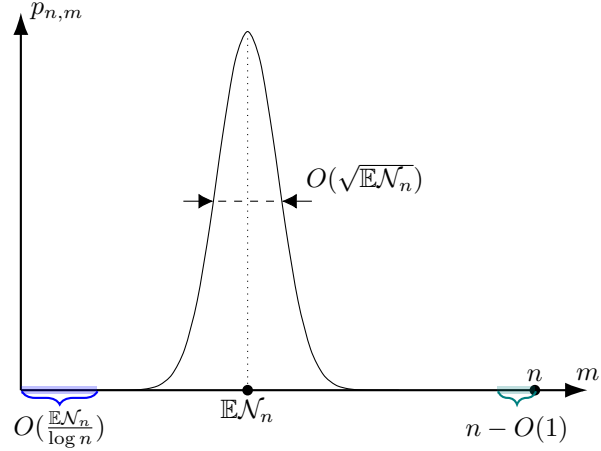
\begin{figure}[t]
    \centering
    \begin{tikzpicture}
        \begin{scope}[scale=1]
            %%%%%%%%%%%%%%%%%%%% Draw axes %%%%%%%%%%%%%%%%%%%%
            \draw[thick, -{Latex[length=3mm, width=2mm]}] (0,0) -- (0, 5) node[right] {$p_{n,m}$};
            \draw[thick, -{Latex[length=3mm, width=2mm]}] (0,0) -- (7.5, 0) node[right, above] {$m$};
            %%%%%%%%%%%%%%%%%%%% Draw Gaussian Distribution %%%%%%%%%%%%%%%%%%%%
            \begin{axis}[
                anchor=origin,
                x=0.4cm,
                at={(3cm,0)},
                style={samples=51,smooth},
                hide axis
            ]
            \addplot[mark=none] {gauss(0,1)};
            \end{axis}
            %%%%%%%%%%%%%%%%%%%% Draw auxillary parts %%%%%%%%%%%%%%%%%%%%
            \draw[-, -{Latex[length=2mm, width=2mm]}] ({1.9+1.95}, 2.5) -- ({1.5+1.95}, 2.5);
            \draw[-, -{Latex[length=2mm, width=2mm]}] ({0.55+1.6}, 2.5) -- ({0.95+1.6}, 2.5); \node[right] at ({1.7+1.95}, 2.8) {$O(\sqrt{\mathbb{E}\mathcal{N}_n})$};
            \draw[dashed] ({0.9+1.5}, 2.5) -- ({1.5+1.95},2.5);

            \draw[dotted] ({1.2+1.8}, 0) -- ({1.2+1.8}, 4.7);
            
            \fill[black] ({1.2+1.8}, 0) circle (2pt) node[below] {$\mathbb{E}\mathcal{N}_n$};
            \fill[black] (6.8, 0) circle (2pt) node[above] {$n$};

            \fill[nearly transparent, blue] (0,-0.05) rectangle (1,0.05);
            \draw [thick, blue,decorate,decoration={brace,amplitude=5pt,mirror},xshift=0.4pt,yshift=-0.4pt](0,0) -- (1,0) node[black,midway,yshift=-0.5cm] {$O(\frac{\mathbb{E}\mathcal{N}_n}{\log n})$};
            \fill[nearly transparent, teal] (6.3,-0.05) rectangle (6.8,0.05);
            \draw [thick, teal,decorate,decoration={brace,amplitude=5pt,mirror},xshift=0.4pt,yshift=-0.4pt](6.3,0) -- (6.8,0) node[black,midway,yshift=-0.5cm] {$n-O(1)$};
        \end{scope}
    \end{tikzpicture}
    \caption{The plot shows a schematic illustration of the graph $m \mapsto p_{n,m}$, along with markings of the left and right tail regimes.} \label{Fig_illustration of regimes}
\end{figure}

The final question of interest concerns the probability of rare events. In our context, one again needs to distinguish further between two cases:  
\begin{itemize}
    \item \textbf{Left tail event}: This corresponds to the case where there is an exceptionally small number of real eigenvalues compared to the typical number. 
    \smallskip 
    \item \textbf{Right tail event}: This is the opposite to the first one, namely, the situation where there is an exceptionally large number of real eigenvalues. 
\end{itemize}
See Figure~\ref{Fig_illustration of regimes} for an illustration of the different regimes, including the Gaussian fluctuations near the typical value, as well as the right and left tail events. 

To describe such rare events, let \( p_{n,m} \) denote the probability that the elliptic GinOE of size \( n \) has exactly \( m \) real eigenvalues.  
The left tail event was addressed in \cite{BMS23}, where it was shown that for an \( m \) small enough compared to the expected number of real eigenvalues--more precisely, \( m = O(\mathbb{E} \mathcal{N}_n/\log n) \)--the following holds: 
\begin{equation} \label{asymp LDP small ev}
\log p_{n,m}    \begin{cases}
       = \mathfrak{d}(\tau) \sqrt{n}+o(\sqrt{n}), &\textup{for } \tau \in [0,1) \text{ fixed}, 
        \smallskip 
        \\
       \le \mathfrak{f}(\alpha) \, n +o(n), &\textup{for } \tau = 1 - \frac{\alpha^2}{n},
    \end{cases}
\end{equation}
where 
\begin{equation}
\mathfrak{d}(\tau) := - \sqrt{\frac{1+\tau}{1-\tau}} \frac{1}{\sqrt{2\pi}} \zeta(\tfrac{3}{2}), \qquad \mathfrak{f}(\alpha) := -\frac{2}{\pi} \int_0^1 \log\Big(1-e^{-\alpha^2 s^2}\Big) \sqrt{1-s^2} \, ds. 
\end{equation}
Here, $\zeta(s)$ is the Riemann-zeta function. We also mention that the result for the GinOE case was previously obtained in \cite{KPTTZ15}, see also \cite{Fo15a}. Additionally, we note that the inequality in \eqref{asymp LDP small ev} for the weakly non-Hermitian regime is expected to be an equality, see \cite{BMS23}.

In contrast to the left tail case, the right tail event remains less understood. In this work, we aim to contribute to this direction.  
In the extreme case \( m = n \), where all eigenvalues are real, a closed-form expression for the probability exists \cite{Ed97, FN08}, which reads
\begin{equation} \label{eq. p n n closed form}
 \log p_{n,n} = -\frac{1}{4} \log\Big( \frac{2}{1+\tau} \Big) n^2 + \frac{1}{4} \log\Big( \frac{2}{1+\tau} \Big) n.   
\end{equation}
This formula is derived from the ratio between the partition functions of the elliptic GinOE and the GOE.  
The only other known case is for the GinOE, where the spectrum contains only a single complex conjugate pair. To be more precise, it was established in \cite[Eq.(1.3c)]{AK07} that  
\begin{equation} \label{eq. p n n-2 for the GinOE} 
\log p_{n,n-2} \Big|_{ \tau=0 } = -\frac{\log 2}{4} n^2 + \Big( \frac{\log 2}{4} + \log 3 \Big) n -\frac12 \log n  + \log \frac{\sqrt{3}}{8\sqrt{ \pi}} + o(1)
\end{equation}
as $n\to \infty$.

Beyond the mentioned results \eqref{eq. p n n closed form} and  \eqref{eq. p n n-2 for the GinOE}, the expansion of $p_{n,n-2l}$ for given finite $l$ has not been explored in the literature, even for the classical GinOE, $\tau=0$ case. In our main result, we address this question both for strong and weak non-Hermiticity. 
In the sequel, we assume that $n$ is an even integer. While the case of odd $n$ is also manageable, it involves additional computations. Notably, the case of odd $n$ has been explored in \cite{Si07, FM09} for the GinOE.

Our first main result is as follows.  

\begin{theorem}[\textbf{The probability of having a finite number $l$ of complex eigenvalue pairs}] \label{Thm. LDP rate function}
Let $n$ be an even integer and $l$ be a fixed nonnegative integer. Then as $n\to\infty$, we have the following.
    \begin{itemize}
        \item[\textup{(i)}] \textbf{\textup{(Strong non-Hermiticity)}}
        For fixed $\tau \in [0,1)$, we have
        \begin{equation} \label{expansion of p n n2l strong}
          \log p_{n,n-2l} = a_1 n^2+a_2 n +a_3 \log n+O(1),
        \end{equation}
        where 
        \begin{equation}
        a_1=  -\frac{1}{4} \log\Big( \frac{2}{1+\tau} \Big), \qquad a_2=  \frac{1}{4} \log\Big( \frac{2}{1+\tau} \Big)  +  l \log \Big(\frac{3-\tau}{1+\tau}\Big) , \qquad a_3=  - \frac{l^2}{2}. 
        \end{equation}
        
        \item[\textup{(ii)}] \textbf{\textup{(Weak non-Hermiticity)}}
        For fixed $\alpha \in [0, \infty)$ and $\tau = 1 - \alpha^2 / n$, we have
        \begin{equation} \label{expansion of p n n2l weak}
         \log p_{n,n-2l} = b_1 n+b_2 \log n+b_3+O(n^{-\epsilon})
        \end{equation}
        for some sufficiently small $\epsilon>0$, where 
        \begin{equation}
        b_1=  -\frac{\alpha^2}{8} , \qquad b_2= l ,\qquad b_3= l \log \Big( \frac{e^{\frac{\alpha^2}{2}} ( I_0(\tfrac{\alpha^2}{2}) - I_1(\tfrac{\alpha^2}{2}) ) - 1}{2} \Big)- \log (l!)  + \frac{\alpha^2}{8}-\frac{\alpha^4}{32}. 
        \end{equation}
    \end{itemize}
\end{theorem}

See Figure~\ref{Fig_numerics main} for the numerical verifications of Theorem~\ref{Thm. LDP rate function}.
We stress that $b_3$ is related to $\mathfrak{c}(\alpha)$ in \eqref{eq. coeff avg number} by
\begin{equation}
    b_3= l \log \Big( \frac{\mathfrak{c}(i \alpha)-1}{2} \Big)- \log (l!)  + \frac{\alpha^2}{8}-\frac{\alpha^4}{32}.
\end{equation} 

\begin{figure}[b]
    \begin{subfigure}{0.42\textwidth}
        \begin{center}
            \includegraphics[width=\linewidth]{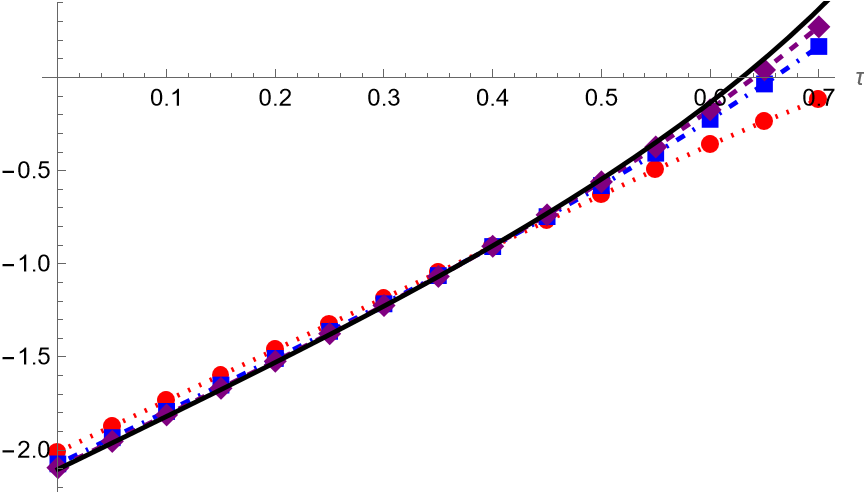}
        \end{center}
        \subcaption{ Strong non-Hermiticity, $l=1$. }
    \end{subfigure}
    \qquad
    \begin{subfigure}{0.42\textwidth}
        \begin{center}
            \includegraphics[width=\linewidth]{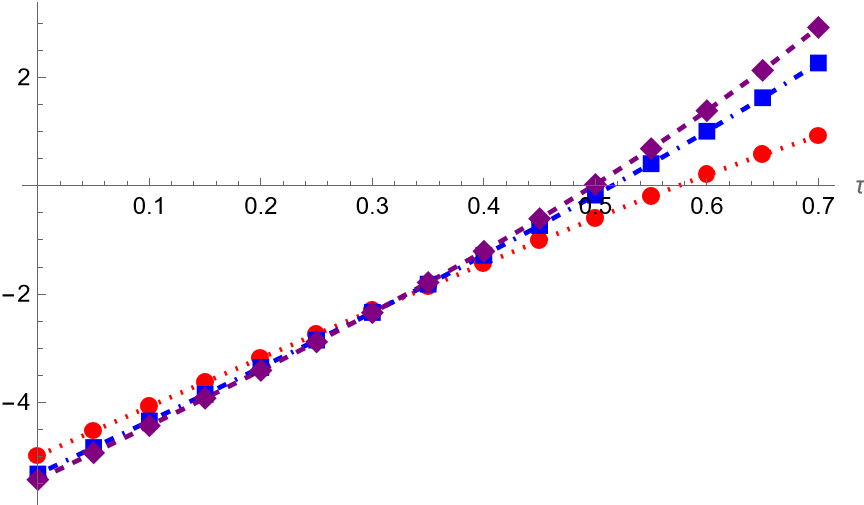}
        \end{center}
        \subcaption{ Strong non-Hermiticity, $l=2$.  }
    \end{subfigure}
    
    \begin{subfigure}{0.42\textwidth}
        \begin{center}
            \includegraphics[width=\linewidth]{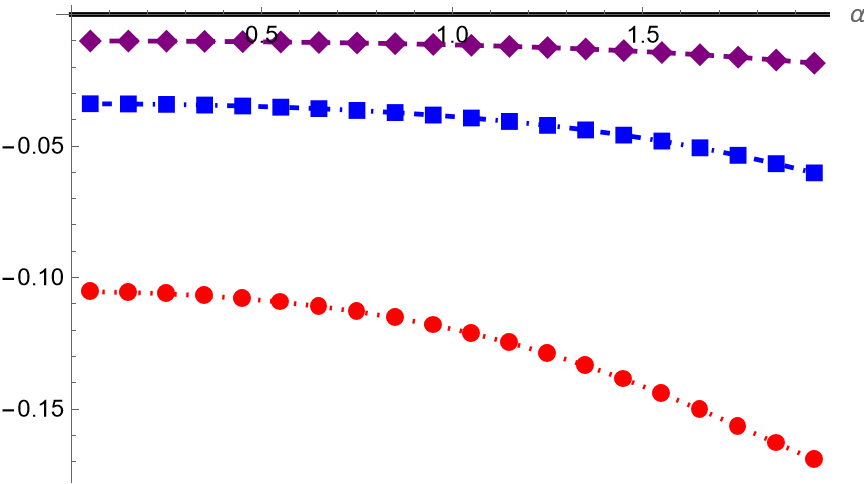}
        \end{center}
        \subcaption{ Weak non-Hermiticity, $l=1$. }
    \end{subfigure}
    \qquad 
    \begin{subfigure}{0.42\textwidth}
        \begin{center}
            \includegraphics[width=\linewidth]{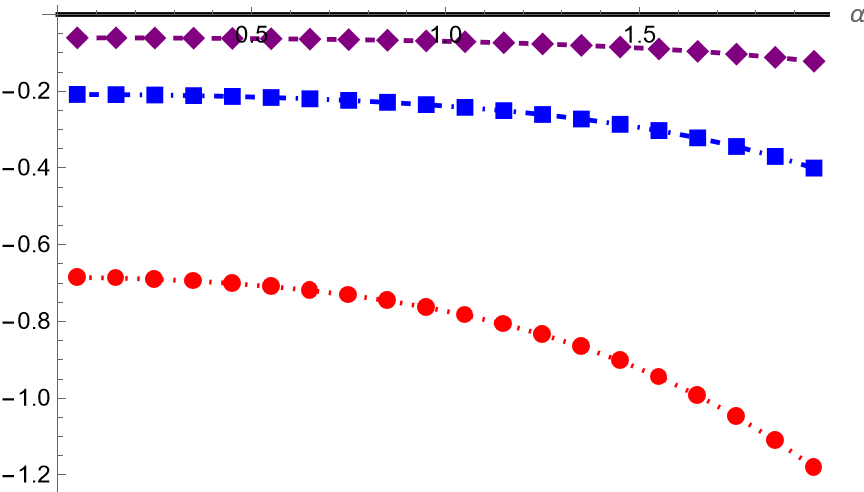}
        \end{center}
        \subcaption{  Weak non-Hermiticity, $l=2$.  }
    \end{subfigure}
    \caption{The plots (A) and (B) illustrate $\tau \mapsto \log p_{n,n-2l} - (a_1 n^2 + a_2 n + a_3 \log n)$ for $n = 10$ (red, dotted), $30$ (blue, dot-dashed), and $100$ (purple, dashed). From these plots, it can be observed that the values remain bounded as $n$ increases. Additionally, plot (A) compares this with $\tau \mapsto \log \big( \frac{(3-\tau)^{1/2} (1+\tau)^{3/2}}{8\sqrt{\pi} (1-\tau)^{3/2}} \big)$ (black, solid), as given in \eqref{eq. LDP sh l=1 v2}.
 The plots (C) and (D) display $\alpha \mapsto \log p_{n,n-2l} - (b_1 n + b_2 \log n + b_3)$ for $n = 10$ (red, dotted), $30$ (blue, dot-dashed), and $100$ (purple, dashed), along with $\alpha \mapsto 0$ (black, solid). These plots demonstrate that the values converge to $0$ as $n$ increases. We used \eqref{eq. pnm zonal formula} for the numerical evaluation at each data point. } \label{Fig_numerics main}
\end{figure}
 
An equivalent way of writing \eqref{expansion of p n n2l strong} is by taking the ratio with the closed formula \eqref{eq. p n n closed form} as follows 
          \begin{equation} \label{eq. LDP sH fixed l}
            \log \Big(\frac{p_{n,n-2l}}{p_{n,n}}\Big) = l \log \Big(\frac{3-\tau}{1+\tau}\Big)  n - \frac{l^2}{2} \log n + O(1) . 
        \end{equation} 
On the other hand, by comparing with the case $l=1$, one can also rewrite \eqref{eq. LDP sH fixed l} as  
        \begin{equation}
        \frac{p_{n,n-2l}}{p_{n,n}} = \frac{1}{n^{l(l-1)/2}} \Big( \frac{p_{n,n-2}}{p_{n,n}} \Big)^l O(1).
        \end{equation}
Similarly, for weak non-Hermiticity, the expansion \eqref{expansion of p n n2l weak} can be rewritten as 
         \begin{equation} \label{eq. LDP wH fixed l}
            \log \Big(\frac{p_{n,n-2l} }{p_{n,n}}\Big) = l \log n + l \log \Big( \frac{e^{\frac{\alpha^2}{2}} ( I_0(\tfrac{\alpha^2}{2}) - I_1(\tfrac{\alpha^2}{2}) ) - 1}{2} \Big)  - \log (l!) + O(n^{-\epsilon}) 
        \end{equation}
        and
        \begin{equation}
        \frac{p_{n,n-2l}}{p_{n,n}} = \frac{1}{l!} \Big( \frac{p_{n,n-2}}{p_{n,n}} \Big)^l (1 + O(n^{-\epsilon})).
    \end{equation}

Observe that the leading order in the expansions \eqref{expansion of p n n2l strong} and \eqref{expansion of p n n2l weak} does not depend on $l$. Such behaviour can be intuitively understood from an electrostatic perspective. More precisely, $\log p_{n,n-2l}$ can be interpreted as the logarithm of the ratio between the partition function of the elliptic GinOE and that of the elliptic GinOE conditioned on having $n-2l$ real eigenvalues. Under this conditional event, the leading-order macroscopic behaviour remains unaffected by the finite parameter $l$. However, this phenomenon no longer holds when $l = O(n)$, as the leading order does depend on the magnitude of $l$. We refer to \cite{MPTW16} for further discussion.

Beyond the leading-order coefficients, the subleading-order coefficients can also be understood from an electrostatic perspective. For example, consider two random $n$-point configurations on the complex plane associated with the Gibbs measure defined by the electrostatic Hamiltonian. One configuration is conditioned to have all $n$ points on the real line, while the other is conditioned to have $n - 2l$ points on the real line and $l$ complex conjugate pairs.  
Then, the term $l \log(\frac{3 - \tau}{1 + \tau})$ in \eqref{eq. LDP sH fixed l} arises from the typical electrostatic energy difference between these two $n$-point configurations.  
See also Lemma~\ref{lem. existence of global minimum for Q_n's}.

\medskip 

Our next result focuses on the specific case where \( l = 1 \). In this case, we are able to derive the full order asymptotic expansion.
For this purpose, we first define 
   \begin{align}
        \begin{split} \label{eq. coeff mathcal A}
            \mathcal{A}_r(\tau) &= \frac{ (1+\tau)^{\frac{3}{2}} }{ 4 \pi^\frac{3}{2} (1-\tau) } \sum_{j=0}^r \Big( \frac{3-\tau}{4} \Big)^{r+\frac{1}{2}} \Big( \frac{2}{1-\tau} \Big)^{j+\frac{1}{2}} \frac{\Gamma(j+\frac{3}{2}) }{ \Gamma(j+1) }
            \\
            & \qquad \times \bigg[ \Gamma(r+\tfrac{1}{2}) + \sum_{q=1}^{r-j} \sum_{k=1}^q \Big( \frac{4}{3-\tau} \Big)^q (-1)^k \frac{\Gamma(r+k+\tfrac{1}{2})}{ \Gamma(k+1) } \sum_{\substack{u_1 + \cdots + u_k = q, \\ u_1, \ldots, u_k \geq 1}} \prod_{v=1}^k \frac{1}{u_v + 1} \bigg].
        \end{split}
        \end{align}
This will be used in the expansion at strong non-Hermiticity. 
Notice that by definition, the first few terms of \eqref{eq. coeff mathcal A} are given by  
\begin{align}
    \begin{split}
        &\mathcal{A}_0(\tau) = \frac{ (3-\tau)^{\frac{1}{2}} (1+\tau)^{\frac{3}{2}} }{8\sqrt{\pi} (1-\tau)^{\frac{3}{2}}},
        \qquad \mathcal{A}_1(\tau) = \frac{ (3-\tau)^{\frac{1}{2}} (1+\tau)^{\frac{3}{2}} (9 - 4\tau + \tau^2) }{ 64 \sqrt{\pi} (1-\tau)^{\frac{5}{2}} },
        \\
        & \mathcal{A}_2(\tau) = \frac{ (3-\tau)^{\frac{1}{2}} (1 + \tau )^{\frac{3}{2}}
        \left(3 \tau ^4-18 \tau ^3+44 \tau ^2-82 \tau
        +143\right)}{512 \sqrt{\pi } (1 - \tau)^{\frac{7}{2}}}.
    \end{split}
    \end{align}
To describe the expansion for weak non-Hermiticity, we define 
  \begin{equation} \label{eq. coeff mathcal B}
        \mathcal{B}_k(\alpha) = \sum_{j=0}^k \binom{\tfrac{1}{2}}{j} \Big( - \frac{\alpha^2}{2} \Big)^j \mathfrak{s}_{k-j}(\alpha),
    \end{equation}
    where  
    \begin{align}
    \begin{split}
        \mathfrak{s}_{k}(\alpha) &= \frac{\sqrt{\pi} (2k - (2k-1) \alpha^2) \Gamma(k-\tfrac{1}{2}) }{ 2^{k+1} \pi \alpha^2 \Gamma(k+1) } + \sum_{j=0}^k \mathfrak{u}_{k,j} \frac{ \Gamma(j- \tfrac{1}{2}) }{ \sqrt{\pi} \alpha^2 \Gamma(j) } \gausF(j-\tfrac{1}{2}, j, \alpha^2)
        \\
        & \quad - \sum_{q=1}^k \sum_{j=0}^{k-q} \frac{\mathfrak{u}_{k-q,j}}{2^q} \frac{ \Gamma(j+\tfrac{1}{2}) \Gamma( q- \tfrac{1}{2} ) }{ \pi \alpha^2 \Gamma(j+q) } \gausF(j+\tfrac{1}{2}, j+q, \alpha^2).
    \end{split}
    \end{align}
    Here, $\gausF$ is Kummer confluent hypergeometric function \cite[Chapter 16]{NIST}
    \begin{equation} \label{eq. def of 1F1}
        \gausF(a,b,z)= \sum_{k=0}^\infty \frac{\Gamma(a+k) \Gamma(b)}{\Gamma(a) \Gamma(b+k)} \frac{z^k}{k!},
    \end{equation}
    and $\mathfrak{u}_{j,s}$ are defined by the coefficients of the polynomial 
    \begin{equation}
        \sum_{s=0}^{j} \mathfrak{u}_{j,s} t^s = \sum_{k=1}^j \frac{\alpha^{2k} }{2^{j+k} k!} \sum_{\substack{j_1 + \cdots + j_k = j, \\ j_1, \ldots, j_k \geq 1}}  \prod_{s=1}^k \frac{1 - (1 - 2t)^{j_s + 1}}{j_s + 1}.
    \end{equation}
  The first few terms are given by 
  \begin{align}
  \begin{split}
      &\mathcal{B}_0(\alpha)= e^{\frac{\alpha^2}{2}} \Big( I_0(\tfrac{\alpha^2}{2}) - I_1(\tfrac{\alpha^2}{2}) \Big) - 1,
        \qquad \mathcal{B}_1(\alpha) = \frac{1}{2}\Big( e^{\frac{\alpha^2}{2}} I_0(\tfrac{\alpha^2}{2}) - \alpha^2 - 1 \Big),
        \\
        & \mathcal{B}_2(\alpha) = \frac{\alpha^2}{24} e^{\frac{\alpha^2}{2}} \Big( \alpha^2 I_0(\tfrac{\alpha^2}{2}) + (\alpha^2-1) I_1(\tfrac{\alpha^2}{2})  \Big),
  \end{split}
    \end{align}
    where we expanded $\gausF$ in terms of the modified Bessel functions of the first kind using \eqref{I nu} and \eqref{eq. def of 1F1}, e.g. \cite[Section 13.6]{NIST}.

\begin{theorem}[\textbf{The probability of having one complex conjugate pair of eigenvalues, i.e. $l=1$}] \label{Thm. LDP rate function l=1}
Let $M$ be a fixed positive integer. Then as $n \to \infty$, we have the following. 
\begin{itemize}
    \item[\textup{(i)}]  \textbf{\textup{(Strong non-Hermiticity)}}
        For fixed $\tau \in [0,1)$, we have 
      \begin{equation}   \label{eq. LDP sH l=1}
            \frac{p_{n,n-2}}{p_{n,n}} =  \frac{1}{\sqrt{n}} \Big(\frac{3-\tau}{1+\tau}\Big)^n \Big(\sum_{k=0}^{M-1} \mathcal{A}_k(\tau) n^{-k} + O(n^{-M}) \Big),
        \end{equation}
        where $\mathcal{A}_k$ is given by \eqref{eq. coeff mathcal A}. In particular, we have 
      \begin{align}
       \begin{split}   \label{eq. LDP sh l=1 v2}
        \log p_{n,n-2}& =  -\frac{1}{4} \log\Big( \frac{2}{1+\tau} \Big) n^2 +\bigg[ \frac{1}{4} \log\Big( \frac{2}{1+\tau}  \Big)+ \log  \Big(\frac{3-\tau}{1+\tau}\Big)   \bigg] n
        \\
        &\quad -\frac12 \log n + \log \Big(  \frac{ (3-\tau)^{\frac{1}{2}} (1+\tau)^{\frac{3}{2}} }{8\sqrt{\pi} (1-\tau)^{\frac{3}{2}}} \Big)+O( n^{-1} ).
        \end{split}
        \end{align}
    \item[\textup{(ii)}] \textbf{\textup{(Weak non-Hermiticity)}} For fixed $\alpha \in [0, \infty)$ and $\tau = 1 - \alpha^2 / n$, we have 
    \begin{equation}  \label{eq. LDP wH l=1}
        \frac{p_{n,n-2}}{p_{n,n}} =  \frac{n}{2} \Big( \sum_{k=0}^{M-1} \mathcal{B}_k(\alpha) n^{-k}  + O(n^{-M}) \Big),
    \end{equation}
    where $\mathcal{B}_k$ is given by \eqref{eq. coeff mathcal B}. In particular, we have 
    \begin{align}
    \begin{split} \label{eq. LDP wh l=1 v2}
     \log p_{n,n-2} &=  -\frac{\alpha^2}{8} n+\log n + \log \Big( \frac{e^{\frac{\alpha^2}{2}} ( I_0(\tfrac{\alpha^2}{2}) - I_1(\tfrac{\alpha^2}{2}) ) - 1}{2} \Big)  + \frac{\alpha^2}{8}-\frac{\alpha^4}{32}
     \\
     &\quad +  \bigg[ \frac12\frac{e^{\frac{\alpha^2}{2}} I_0(\tfrac{\alpha^2}{2}) - \alpha^2 - 1 }{e^{\frac{\alpha^2}{2}} ( I_0(\tfrac{\alpha^2}{2}) - I_1(\tfrac{\alpha^2}{2}) ) - 1} + \frac{\alpha^4}{32}-\frac{\alpha^6}{96} \bigg] n^{-1}  +O(n^{-2}).
    \end{split}
    \end{align}
\end{itemize} 
\end{theorem}

Notice in particular that \eqref{eq. LDP sh l=1 v2} recovers the previous finding \eqref{eq. p n n-2 for the GinOE} for the extremal case $\tau=0.$ 
We mention that Theorem~\ref{Thm. LDP rate function l=1} is one of the very few cases where the full order expansion is tractable in an explicit way.

We emphasise that in Theorems~\ref{Thm. LDP rate function} and~\ref{Thm. LDP rate function l=1}, the strong and weak non-Hermiticity regimes should be treated separately, as no uniform asymptotic behaviour holds over the entire range \( \tau \in [0,1] \). Nonetheless, the leading-order behaviours can be formally connected—as \( \tau \to 1 \) from the strong non-Hermiticity side and \( \alpha \to \infty \) from the weak non-Hermiticity side—in a manner analogous to the discussion in \cite[Section 1.1]{BKLL23} concerning the mean and variance of the number of real eigenvalues.

To be more precise, assume $ \alpha = n^\epsilon $ for sufficiently small $ \epsilon > 0 $. Then, one can verify that the right-hand side of \eqref{eq. LDP sH l=1} exhibits the asymptotic behaviour
\begin{equation*} 
    \frac{n}{2\sqrt{\pi} \alpha^3} e^{\alpha^2} \Big(1 + O(n^{-4\epsilon}) \Big),
\end{equation*}
On the other hand, using the asymptotic expansions of the modified Bessel functions of the first kind,
\begin{equation*}
    I_0(x) \sim \frac{e^x}{\sqrt{2\pi x}} \Big(1 + \frac{1}{8x} + \frac{9}{128 x^2} \Big), \qquad I_1(x) \sim \frac{e^x}{\sqrt{2\pi x}}\Big(1 - \frac{3}{8x} - \frac{15}{128 x^2} \Big), \qquad \text{as} \quad x \to \infty,
\end{equation*}
it follows that the right-hand side of \eqref{eq. LDP wH l=1} satisfies the asymptotic behaviour
\begin{equation*}
 \frac{n}{2\sqrt{\pi} \alpha^3} e^{\alpha^2} \Big(1 + 
 O(n^{-2\epsilon}) \Big). 
\end{equation*}
These expressions show that the two asymptotic behaviours agree at least at the leading order.

\subsection*{Organisation of the paper}
The rest of this paper is organised as follows. In Section~\ref{Section_integrable structure}, we provide some preliminaries for our analysis, including the Pfaffian integration formula for \(p_{n,m}\). Section~\ref{Section_one complex conjudate} focuses on the special case when there is only one complex conjugate pair, and we prove Theorem~\ref{Thm. LDP rate function l=1}. In Section~\ref{Sec. l finite sH}, we establish Theorem~\ref{Thm. LDP rate function} (i) for the strongly non-Hermitian regime using a mean-field approximation and potential-theoretic approach. In Section~\ref{Sec. l finite wH}, we prove Theorem~\ref{Thm. LDP rate function} (ii) for the weakly non-Hermitian regime through asymptotic analysis involving skew-orthogonal polynomials.

\section{Integrable structure and Pfaffian formula of the elliptic GinOE} \label{Section_integrable structure}

In this section, we recall the Pfaffian formula for the probability $p_{n,m}$.
The formula was derived in \cite{AK07} for $\tau=0$ using \emph{the Pfaffian integration Theorem}, see also \cite{BK07}. 
With minor modifications, the derivation of the formula for general $\tau \in [0,1)$ is similar to that for $\tau = 0$.
We give an idea of the derivation in this section, and refer to \cite[Sections 3 and 6]{AK07} for details.

Let $\vmu = (\mu_1, \ldots, \mu_m) \in \R^m$ and $\vzeta = (\zeta_1, \ldots, \zeta_l) \in \HH^l$, where $\HH$ is the upper-half plane.  
Then, the partial joint probability density function of the elliptic GinOE, with a  number $m$ of 
real eigenvalues $\vmu$ and $l$ of complex eigenvalue pairs $\vzeta$ and $\overline{\vzeta}$, can be found in \cite{FN08}.
For the latter purpose, it is often convenient to rescale eigenvalues $$\vlambda = \sqrt{\frac{1+\tau}{n}}\vmu, \qquad \vz = \sqrt{\frac{1+\tau}{n}}\vzeta.$$ 
To describe the joint probability distribution, it is convenient to introduce the following three Hamiltonian functions that capture pairwise interactions among the $n$ eigenvalues and the influence of an external potential: 
\begin{align}
    \mathsf{H}_{11}(\vlambda) &= \frac{1}{n}\sum_{1\leq j < k \leq m} \log \frac{1}{|\lambda_j - \lambda_k|} + \sum_{j=1}^m \frac{\lambda_j^2}{2(1+\tau)},  \label{def of Ham H11}
    \\
    \mathsf{H}_{12}(\vlambda, \vz) &= \frac{1}{n} \sum_{j=1}^m \sum_{k=1}^l \log \frac{1}{|\lambda_j - z_k|^2},  \label{def of Ham H12}
    \\
    \mathsf{H}_{22}(\vz) &= \frac{1}{n}\sum_{1\leq j < k \leq l} \log \frac{1}{|z_j - z_k|^2 |z_j-\Bar{z}_k|^2} + \sum_{j=1}^l Q_{\tau,n}(z_j),  \label{def of Ham H22}
\end{align}
where
\begin{align} \label{eq. Q_{tau,n}}
\begin{split}
    Q_{\tau,n}(z) &= - \frac{1}{n} \log\Big[ |z-\Bar{z}| \erfc\Big( \frac{\sqrt{n} (z - \overline{z})}{i\sqrt{2{(1-\tau^2)}}} \Big) \Big] + \frac{z^2 + \overline{z}^2}{2(1+\tau)}.
\end{split}
\end{align}
Then the joint distribution is given by 
\begin{align} \label{scaled p.d.f.}
\begin{split}
    \mathcal{P}_{m,l}(\vlambda; \vz) &= \frac{1}{Z_{m,l}} e^{-n (\mathsf{H}_{11}(\vlambda) + \mathsf{H}_{12}(\vlambda, \vz) + \mathsf{H}_{22}(\vz))},
\end{split}
\end{align}
where 
\begin{equation} \label{eq. def of widetilde Z_ml}
Z_{m,l} = 2^{\frac{n(n+1)}{4} - l} n^{-\frac{n(n+1)}{4}} m! \, l! \,(1+\tau)^{\frac{n}{2}} \prod_{j=1}^{n} \Gamma(j/2).
\end{equation}
By definition, the probability $p_{n,m}$ to have $m$ real eigenvalues and $l$ complex conjugate pairs is given by
\begin{align}
    p_{n,m} = \frac{1}{Z_{m,l}} \int_{\HH^l} \int_{\R^m} e^{-n (\mathsf{H}_{11}(\vlambda) + \mathsf{H}_{12}(\vlambda, \vz) + \mathsf{H}_{22}(\vz))} d\vlambda \, d^2\vz. \label{eq. p_nm integral vz}
\end{align}

The formalism of skew-orthogonal polynomials is crucial for the derivation of the Pfaffian formula, see \cite[Section 3]{AK07}.
For this, recall that the GOE skew product $\langle\cdot,\cdot\rangle$ is defined by 
\begin{equation}
    \langle f, g \rangle = \frac{1}{2} \int_\R d\xi \, e^{-\xi^2/2} \int_\R d\eta \, e^{-\eta^2/2} \sgn(\eta-\xi) f(\xi) g(\eta).
\end{equation} 
We define polynomials 
\begin{equation} \label{eq. skew-ortho polys}
q_{2j}(\xi) = \frac{1 }{2^{2j}} H_{2j}(\xi), \qquad q_{2j+1}(\xi) = \frac{1}{2^{2j}} ( H_{2j}(\xi) - 4j H_{2j-1}(\xi)),
\end{equation}
where 
$$H_k(\xi) = (-1)^k e^{\xi^2} \frac{d^k}{d\xi^k} e^{-\xi^2}$$
is the Hermite polynomial. 
They form skew-orthogonal polynomials with respect to the GOE skew product, i.e.  
\begin{align}
\langle q_{2j}, q_{2j} \rangle = \langle q_{2j+1}, q_{2j+1} \rangle = 0,\qquad     \langle q_{2j}, q_{2k+1} \rangle = - \langle q_{2j+1}, q_{2k} \rangle = h_j \delta_{jk}, \qquad   h_j = \frac{\sqrt{\pi} (2j)!}{2^{2j}}.
\end{align} 
These are building blocks to define prekernels
\begin{align} \label{eq. def of prekernel kappa}
    \kappa_{n}(\zeta,\eta) = \frac{1}{2} e^{-(\zeta^2+\eta^2)/2} \sum_{j=0}^{n/2-1} \frac{q_{2j+1}(\zeta) q_{2j}(\eta) - q_{2j}(\zeta) q_{2j+1}(\eta)}{h_j}.
\end{align}

A key idea of the derivation for the Pfaffian formula is identifying the integrand in \eqref{eq. p_nm integral vz} depending on $\vlambda$ as a product of characteristic polynomials of an $m\times m$ GOE.
The average of a product of GOE characteristic polynomials can be expressed in terms of Pfaffian of the prekernels in \eqref{eq. def of prekernel kappa} \cite{BS06}. 
Then we have
\begin{equation} \label{eq. p(n,m) pfaffian integral representation}
    p_{n,m} = \frac{p_{n,n}}{l!}\Big(\frac{2}{i}\Big)^{l} \int_{\HH^l} d^2\vzeta \, 
    \Pf\begin{bmatrix}
        \kappa_n(\zeta_j, \zeta_k) & \kappa_n(\zeta_j, \overline{\zeta}_k) 
        \smallskip 
        \\
        \kappa_n(\overline{\zeta}_j, \zeta_k) & \kappa_n(\overline{\zeta}_j, \overline{\zeta}_k)
    \end{bmatrix}_{j,k=1}^l
    \prod_{j=1}^{l} \erfc\Big( \frac{\zeta_j - \overline{\zeta}_j}{i\sqrt{2{(1-\tau)}}} \Big).
\end{equation}
The integral representation \eqref{eq. p(n,m) pfaffian integral representation} can be further evaluated using \emph{Pfaffian integration theorem} \cite{AK07, BK07}. To illustrate the result, we recall the generalised Laguerre polynomials and the zonal polynomials.
For $a\in\R$, the generalised Laguerre polynomials $L_k^{a}(\xi)$ are defined by
\begin{equation}
    L_k^{a}(\xi) = \frac{1}{k!} \xi^{-a} e^{\xi} \frac{d^k}{d\xi^k}(\xi^{k+a} e^{-\xi}).
\end{equation}
The zonal polynomials are defined for nonnegative integers $k$, and can be written in terms of a sum over all partitions of $k$ \cite{Ma15}. Let us denote by $\sigma = (1^{\sigma_1}, \ldots, k^{\sigma_k})$ a partition of $k$. That is $|\sigma| = k$, where $|\sigma| = \sum_{j=1}^k j \sigma_j$. Then the zonal polynomial associated to $(1^k)$ is defined by 
\begin{equation}
    Z_{(1^k)}(\xi_1, \ldots, \xi_k) = (-1)^k k!\sum_{|\sigma| = k} \prod_{j=1}^k \frac{1}{\sigma_j !} \Big( -\frac{\xi_j}{j} \Big)^{\sigma_j}.
\end{equation}
In particular, when $k=1$, we have $ Z_{(1^1)}(\xi) = \xi$.
We define an $n \times n$ matrix $\rhohat_n$ whose $(j,k)$ entries $ [\rhohat_n]_{j,k}$ are given as 
\begin{align}
    [\rhohat_n]_{j,k}  = \int_0^\infty y^{2(k-j)-1} e^{y^2} \erfc\Big(y\sqrt{\frac{2}{1-\tau}}\Big) \bigg[ (2j-1) L_{2j-1}^{2(k-j)-1}(-2y^2) + 2y^2 L_{2j-3}^{2(k-j)+1}(-2y^2) \bigg] \, dy.
\end{align} 
Then, it follows from the computations in \cite[Appendices 3 and 4]{AK07} that the integral \eqref{eq. p(n,m) pfaffian integral representation} is evaluated as follows.

\begin{lemma}
Let $n=m+2l$. Then we have 
\begin{equation} \label{eq. pnm zonal formula}
    p_{n,m} = \frac{p_{n,n}}{l!} Z_{(1^l)}\Big( \Tr(\rhohat_n^1), \ldots, \Tr(\rhohat_n^l)  \Big).
\end{equation}
In particular, for $n=m+2$, i.e. $l=1$, the formula further reduces to
\begin{equation} \label{eq. p(n,n-2) in terms of p(n,n)}
    p_{n,n-2}  = p_{n,n} \int_0^\infty dy \,2 y \, e^{y^2} \erfc\Big( y \sqrt{\frac{2}{1-\tau}}\Big) L_{n-2}^2(-2y^2).
\end{equation}
\end{lemma}
This lemma plays an important role in the asymptotic analysis in the sequel. 
In particular, \eqref{eq. p(n,n-2) in terms of p(n,n)} can be applied to show Theorem~\ref{Thm. LDP rate function l=1}, which will be the focus of the next section.

\section{Special case: one complex conjugate pair} \label{Section_one complex conjudate}

In this section, we consider the probability $p_{n,n-2}$ that there exists exactly one pair of complex eigenvalues, and provide the proof of Theorem~\ref{Thm. LDP rate function l=1}. 

The integral representation \eqref{eq. p(n,n-2) in terms of p(n,n)} is closely related to Gauss' hypergeometric function defined by
\begin{equation}
    \gaussF(a, b; c; z) = \frac{\Gamma(c)}{\Gamma(a) \Gamma(b) } \sum_{k=0}^\infty \frac{\Gamma(a+k) \Gamma(b+k)}{\Gamma(c+k) k!} z^k
\end{equation}
in $|z| < 1$, and by analytic continuation elsewhere.
We begin by changing \eqref{eq. p(n,n-2) in terms of p(n,n)} into a more suitable expression for asymptotic analysis.

\begin{lemma}\label{lem. Integral rep of p n n-2 without special func}
    We have
    \begin{equation} \label{eq. integral rep of p(n,n-2)}
        \frac{p_{n,n-2}}{p_{n,n}} = \frac{ (1+\tau)^{\frac{3}{2}} }{ 4\sqrt{2} \pi (1-\tau) } \int_0^1 \frac{s^{-\frac{1}{2}} (1-s)^{-\frac{3}{2}}}{ 1 - \tfrac{1-\tau}{2} s  } \bigg[ \Big( 1 + \tfrac{2(1-\tau)}{1+\tau} (1-s) \Big)^n - n \tfrac{2(1-\tau)}{1+\tau} (1-s) - 1 \bigg] \, ds.
    \end{equation}
\end{lemma}
\begin{proof}
    Recall that the Laguerre polynomial can be explicitly written as 
    \begin{equation*}
        L_{n}^{a}(x) = \sum_{k=0}^{n} (-1)^k \binom{n+a}{n-k} \frac{x^k}{k!}.
    \end{equation*}
    Using this and \eqref{eq. p(n,n-2) in terms of p(n,n)}, we have 
    \begin{equation*}
        \frac{p_{n,n-2}}{p_{n,n}} = \sum_{k=0}^{n-2} \int_0^\infty e^{y^2} \erfc\Big( y \sqrt{\frac{2}{1-\tau}} \Big) \, 2^{k+1} \binom{n}{k+2} \frac{y^{2k+1}}{k!} \,dy.
    \end{equation*}
    We use an integral formula of the complementary error function \cite[Section 4.3 Eq.(9)]{GN71}
    \begin{equation*}
        \int_0^\infty \erfc(ax) e^{b^2 x^2} x^p \, dx = \frac{\Gamma(\tfrac{1}{2} p + 1)}{\sqrt{\pi} (p+1) a^{p+1}} \gaussF(\tfrac{p+1}{2}, \tfrac{p+2}{2}; \tfrac{p+3}{2}; \tfrac{b^2}{a^2}),
    \end{equation*}
   valid for $\re b^2 < \re a^2, \ \re p > -1$. Then, we obtain
    \begin{align*}
        \frac{p_{n,n-2}}{p_{n,n}} &= \frac{1}{2\sqrt{\pi}} \sum_{k=0}^{n-2} \binom{n}{k+2} \frac{\Gamma(k+\tfrac{3}{2})}{(k+1)!} ( 1-\tau)^{k+1} \gaussF(k+1, k+\tfrac{3}{2}; k+2; \tfrac{1-\tau}{2})
        \\
        &=\frac{1-\tau}{\sqrt{2\pi(1+\tau)}} \sum_{k=0}^{n-2} \binom{n}{k+2} \frac{\Gamma(k+\tfrac{3}{2})}{(k+1)!} \Big( \frac{ 2(1-\tau)}{1+\tau} \Big)^{k} \gaussF(1, \tfrac{1}{2}; k+2; \tfrac{1-\tau}{2}),
    \end{align*}
    where we used a transformation of the hypergeometric function \cite[Eq.(15.8.1)]{NIST} for the second equality.
    Furthermore, by using the well-known Euler integral formula (see e.g. \cite[Eq.(15.6.1)]{NIST}), we have
    \begin{align*}
        \frac{p_{n,n-2}}{p_{n,n}} &= \frac{1-\tau}{\pi \sqrt{2(1+\tau)}} \int_0^1 \frac{s^{-\frac{1}{2}} (1-s)^{\frac{1}{2}} }{ 1 - \tfrac{1-\tau}{2} s  } \sum_{k=0}^{n-2} \binom{n}{k+2} \Big( \frac{ 2(1-\tau)(1-s) }{1+\tau} \Big)^{k} \, ds.
    \end{align*}
    The summation can be further simplified using the binomial theorem, which gives the desired result.
\end{proof}

We shall use the following elementary lemma. 

\begin{lemma} \label{lem. Series expansion poly to exp sH}
    Let $M$ be a nonnegative integer. Then as $N \to \infty$, we have the following. 
    \begin{itemize}
        \item For given $x \in \mathbb{R}$, we have 
         \begin{equation} \label{eq. asymp expansion of (1-x/n)^n}
        \Big( 1 - \frac{x}{N} \Big)^N = e^{-x} \bigg[ \sum_{ q = 0}^M \mathcal{F}_q(x) \Big( \frac{x}{N} \Big)^q + O(N^{-M-1}) \bigg], 
    \end{equation}
    where $\mathcal{F}_0(x) = 1$ and
    \begin{equation}
        \mathcal{F}_q(x) = \sum_{k=1}^q  \frac{(-x)^k}{k!} \sum_{\substack{j_1 + \cdots + j_k = q, \\ j_1, \ldots, j_k \geq 1}} \prod_{s=1}^k \frac{1}{(j_s+1)}, \qquad \textup{for } q \ge 1.  
    \end{equation} 
    \item For given $t \in \mathbb{R}$, we have 
     \begin{equation} \label{eq. power series exp(f(x)) wH}
        \Big( 1 + \frac{2\alpha^2 t}{2N-\alpha^2} \Big)^N = e^{\alpha^2 t} \bigg[ \sum_{q = 0}^M \mathcal{G}_q(t) \Big( \frac{\alpha^2}{N} \Big)^q + O(N^{-M-1}) \bigg],
    \end{equation}
    where $\mathcal{G}_0(t) = 1$ and
    \begin{align}
        \mathcal{G}_q(t) &= \sum_{k=1}^q \frac{\alpha^{2k} }{2^{q+k} k!} \sum_{\substack{j_1 + \cdots + j_k = q, \\ j_1, \ldots, j_k \geq 1}}  \prod_{s=1}^k \frac{1 - (1 - 2t)^{j_s + 1}}{j_s + 1} , \qquad \textup{for } q \ge 1.  
    \end{align}
    \end{itemize}
\end{lemma}
\begin{proof}
    Note that by expanding the logarithm, we have 
    \begin{align*}
        \Big( 1 - \frac{x}{N} \Big)^N &= \exp\Big(- x - \sum_{j = 1}^\infty \frac{x}{j+1} \Big( \frac{x}{N} \Big)^j \Big).
    \end{align*}
   Then the expansion \eqref{eq. asymp expansion of (1-x/n)^n} follows from 
    \begin{equation} \label{eq. power series exp(f(x)) sH}
        \exp\Big( \sum_{j=0}^\infty a_j x^j \Big) = \sum_{j=0}^\infty b_j x^j, \qquad     b_q = \sum_{k=0}^q \sum_{\substack{j_1 + \cdots + j_k = q, \\ j_1, \ldots, j_k \geq 1}} \frac{1}{k!} \prod_{s=1}^k a_{j_s}. 
    \end{equation}

    The second assertion is proved similarly. We use series expansions of $1/(2N-\alpha^2)$ and logarithmic function to obtain
    \begin{align*}
        \log\Big( 1 + \frac{2\alpha^2 t}{2N - \alpha^2} \Big) &= \log\Big( 1 + \sum_{j \geq 1} \Big(\frac{\alpha^2}{2N} \Big)^j \frac{\alpha^2 t}{N} \Big)
        \\
        &= \sum_{k \geq 1} (-1)^{k-1} \frac{1}{k} \Big( \sum_{j \geq 1} \Big(\frac{\alpha^2}{2N} \Big)^j \frac{\alpha^2 t}{N} \Big)^k = \sum_{q \geq 1} t \, \mathfrak{G}_q(t)  \Big(\frac{\alpha^2}{N}\Big)^q
    \end{align*}
    with
    \begin{equation*}
        \mathfrak{G}_j(t) := \sum_{k=0}^{j-1} \frac{(-1)^{k}}{2^{j-k-1}} \frac{1}{k+1} \binom{j-1}{k} t^{k} = \frac{ 1 - (1-2t)^j }{ 2^j j t }.
    \end{equation*}
    Thus, we have
    \begin{align*}
        \Big( 1 + \frac{2\alpha^2 t}{2N-\alpha^2} \Big)^N &= \exp\Big( \alpha^2 t + \sum_{j=1}^\infty \alpha^2 t \, \mathfrak{G}_{j+1}(t) \Big( \frac{\alpha^2}{N}\Big)^j \Big).
    \end{align*}
    Then the desired asymptotic expansion \eqref{eq. power series exp(f(x)) wH} follows from \eqref{eq. power series exp(f(x)) sH}.
\end{proof}

We are now ready to prove Theorem~\ref{Thm. LDP rate function l=1}. 

\begin{proof}[Proof of Theorem~\ref{Thm. LDP rate function l=1}]
We first consider the strongly non-Hermitian regime. 
Notice that for given $\tau \in [0,1)$, the integrand in \eqref{eq. integral rep of p(n,n-2)} is exponentially large near $s=0$.
Thus for any $\epsilon \in (0,1)$, there exists some $\delta>0$ such that
\begin{align*}
   \frac{p_{n,n-2}}{p_{n,n}} &= \frac{ (1+\tau)^{\frac{3}{2}} }{ 4\sqrt{2} \pi (1-\tau) } \int_0^\epsilon \frac{s^{-\frac{1}{2}} (1-s)^{-\frac{3}{2}}}{ 1 - \tfrac{1-\tau}{2} s  } \Big( 1 + \tfrac{2(1-\tau)}{1+\tau} (1-s) \Big)^n  ds \, \Big( 1 + O(e^{-\delta n}) \Big)
   \\
   &= \frac{ (1+\tau)^{\frac{3}{2}} }{ 4\sqrt{2} \pi (1-\tau) } \frac{1}{\sqrt{n}} \Big( \frac{3-\tau}{1+\tau} \Big)^n \int_0^{\epsilon n} \frac{s^{-\frac{1}{2}} (1-\tfrac{s}{n})^{-\frac{3}{2}}}{ 1 - \tfrac{1-\tau}{2n} s  } \Big( 1 - \frac{2(1-\tau)}{3-\tau} \frac{s}{n} \Big)^n  ds \, \Big( 1 + O(e^{-\delta n}) \Big).
  \end{align*}
   By using the binomial series expansion and \eqref{eq. asymp expansion of (1-x/n)^n}, we have  
    \begin{align*}
        \frac{(1-\tfrac{s}{n})^{-\frac{3}{2}}}{ 1 - \tfrac{1-\tau}{2n} s  } \Big( 1 - \frac{2(1-\tau)}{3-\tau} \frac{s}{n} \Big)^n &= \sum_{k \geq 0} \Big( \frac{1-\tau}{2} \Big)^k \Big( \frac{s}{n} \Big)^k \sum_{j \geq 0} \sqrt{\frac{2}{\pi}} \frac{ \Gamma(j+\frac{3}{2}) }{ \Gamma(j+1) } \Big( \frac{s}{n} \Big)^j \sum_{q \geq 0} e^{- a_\tau s} \mathcal{F}_q(a_\tau s) \Big( \frac{a_\tau s}{n} \Big)^q
        \\
        &= \sqrt{\frac{2}{\pi}} e^{- a_\tau s} \sum_{r \geq 0} A_r(s) \Big( \frac{s}{n} \Big)^{r},
    \end{align*}
    where  
    \begin{equation*}
     a_\tau = \frac{2(1-\tau)}{3-\tau}, \qquad    A_r(s) = \sum_{j = 0}^r \sum_{q=0}^{r-j} \Big( \frac{1-\tau}{2} \Big)^{r-j-q} \frac{ \Gamma(j+\frac{3}{2}) }{ \Gamma(j+1) } \mathcal{F}_q(a_\tau s) a_\tau^q.
    \end{equation*}
    Since $A_r(s)$ is a polynomial in $s$, we can change the range of integration from $[0,\epsilon n]$ to $[0,\infty)$ by introducing an exponentially small error.
    Therefore, for some $\delta'>0$, we have
    \begin{align*}
        \frac{p_{n,n-2}}{p_{n,n}} &= \frac{1}{\sqrt{n}} \Big( \frac{3-\tau}{1+\tau} \Big)^n \sum_{r \geq 0} \mathcal{A}_r(\tau)   \frac{1}{n^r}  \Big( 1 + O(e^{-\delta' n}) \Big),
    \end{align*}
    where
    \begin{equation}
        \mathcal{A}_r(\tau) := \frac{ (1+\tau)^{\frac{3}{2}} }{ 4 \pi^\frac{3}{2} (1-\tau) } \int_0^\infty s^{r-\frac{1}{2}} e^{-a_\tau s} A_r(s) \, ds.
    \end{equation}
    Then, straightforward computations show that it can be evaluated as \eqref{eq. coeff mathcal A}, which completes the first part of the theorem.

    Next, we consider the weakly non-Hermitian regime. Substituting $\tau = 1 - \alpha^2/n$ and changing variable $s \mapsto 1-s$, we rewrite \eqref{eq. integral rep of p(n,n-2)} as
    \begin{equation*}
        \frac{p_{n,n-2}}{p_{n,n}} = \frac{n}{2\pi \alpha^2} \Big( 1 - \frac{\alpha^2}{2n} \Big)^{\frac{1}{2}} \int_0^1 \frac{(1-s)^{-\frac{1}{2}} s^{-\frac{3}{2}}}{  1 + \frac{\alpha^2 s}{2n - \alpha^2} } \bigg[ \Big( 1 + \frac{2\alpha^2 s}{2n - \alpha^2} \Big)^n - n \frac{2\alpha^2}{2n - \alpha^2} s -1\bigg] \, ds.
    \end{equation*}
    Using \eqref{eq. power series exp(f(x)) wH} and the series expansions
    \begin{align*}
        \frac{1}{1 + \frac{\alpha^2 s}{2n - \alpha^2}} = 1 - \sum_{k=0}^\infty s (1-s)^k \Big( \frac{\alpha^2}{2n} \Big)^{k+1}, \qquad  \frac{\alpha^2 s}{1 - \frac{\alpha^2 s}{2n}} = \alpha^2 s \sum_{k=0}^\infty \Big( \frac{\alpha^2}{2n} \Big)^k,
    \end{align*}
    we obtain
    \begin{equation*}
        \frac{p_{n,n-2}}{p_{n,n}} = \frac{n}{2} \Big( 1 - \frac{\alpha^2}{2n} \Big)^{\frac{1}{2}} \sum_{k \geq 0} \mathfrak{s}_k(\alpha) \Big( \frac{\alpha^2}{n} \Big)^k,
    \end{equation*}
    where
    \begin{equation}
        \mathfrak{s}_0(\alpha) := \frac{1}{\pi \alpha^2} \int_0^1 t^{-\frac{3}{2}} (1-t)^{-\frac{1}{2}} ( e^{\alpha^2 t} - \alpha^2 t - 1 ) \, dt = e^{\frac{\alpha^2}{2}} (I_0(\tfrac{\alpha^2}{2}) + I_1(\tfrac{\alpha^2}{2})) -1 .
    \end{equation}
    and
    \begin{equation}
        \mathfrak{s}_k(\alpha) := 
            \frac{1}{\pi \alpha^2} \int_0^1 t^{-\frac{3}{2}} (1-t)^{-\frac{1}{2}} \bigg( e^{\alpha^2 t} \Big[ \mathcal{G}_k(t) - \sum_{r=1}^k \mathcal{G}_{k-r}(t) \frac{t(1-t)^{r-1}}{2^r} \Big] - \frac{\alpha^2(1-t) - 1}{2^k} t(1-t)^{k-1}  \bigg) \, dt
    \end{equation}
    for $k \geq 1$. Then, by expanding $(1- \frac{\alpha^2}{2n})^{-\frac{1}{2}}$, we conclude \eqref{eq. LDP wH l=1}.
\end{proof}

\section{General case: strong non-Hermiticity} \label{Sec. l finite sH}

In this section, we prove Theorem~\ref{Thm. LDP rate function} (i). As previously mentioned, we shall apply a potential theoretic approach.  
For this purpose, we define 
\begin{align}
\mathsf{H}_{\rm sc}(\vz) = \sum_{k=1}^l Q_{\rm sc}(z_k), \qquad  Q_{\rm sc}(z) = \int_{\R} \log \frac{1}{|z-t|^2} \, d\rho_{\rm sc}(t), \qquad \rho_{\rm sc}(t) =   \frac{ \sqrt{2(1+\tau) - t^2} }{\pi (1+\tau)}\,  \mathbbm{1}_{ \{ t^2 \leq 2(1+\tau) \}  }.
\end{align}
Here, the subscript ``$\rm sc$'' stands for ``semicircle''.
The overall strategy of our proof can be summarised as follows.
\begin{itemize}
    \item We construct elementary bounds for $p_{n,m}$ in Lemma~\ref{lem. bounds on p_nm: 1}, and approximate them using the \emph{strong Szeg\H{o} theorem} as demonstrated in Lemma~\ref{lem. bounds on p_nm: 2}. The approximated bounds can be expressed in terms of the potential $Q_n^{(r)}$ in \eqref{eq. def of Q_n^(r)}.
    \smallskip  
    \item Next, we analyse the approximated potential $Q^{(r)}_n$, and show that it attains its global minimum in $\HH$ for all $n$ (Lemma~\ref{lem. existence of global minimum for Q_n's}). Furthermore, we compute the precise location and corresponding value with $O(n^{-1})$ accuracy.
    \smallskip   
    \item We analyse the integral with exponential weights in Lemmas~\ref{lem. apprx int_H^l to int_K^l} and \ref{lem. asymp vandermonde integral}. Using these results, we show that the bounds for $p_{n,m}$ obtained in the previous step converge to the same value up to a multiplicative bounded constant. This establishes Theorem~\ref{Thm. LDP rate function} for the case of strong non-Hermiticity.
\end{itemize}

\subsection{Mean field approximation of the real eigenvalue distribution}

It is convenient to introduce the following notation: for a constant $r\geq 0$ and $\vz = (x_1+i y_1, \ldots, x_l + i y_l) \in \HH^l$, 
\begin{align}
    \HH^l_{r} = \{ \vz^{(r)} : \vz \in \HH^l \}, \qquad \vz^{(r)} = (x_1 + i \max\{y_1, r\}, \ldots, x_l + i \max\{y_l, r\}).
\end{align}
In Lemma~\ref{lem. bounds on p_nm: 1} below, we derive bounds for the probability $p_{n,m}$ by truncating the domains of the integrals into subdomains where the integrand remains regular for our analysis. Similar results can be found in the literature, see e.g. \cite[Eq.(2.1)]{Shc11}, \cite[Eq.(4.26)]{Joh98}, and \cite[Lemma 1]{BPS95}.

We proceed our analysis based on a statement given in \cite[Lemma 1]{BPS95} which is formulated as follow. Suppose a function $V: \R \to \R$ is bounded below, and assume that there exist constants $\epsilon>0$ and $s > 2e$ such that 
$$ 
V(\lambda) - \max_{\lambda' \in [-e,e]} V(\lambda') \geq (2+\epsilon) \log |\lambda| 
$$ 
for any $|\lambda| > s$. 
In terms of the Hamiltonian  
\begin{equation}
    \mathsf{H}_N(\lambda_1, \ldots ,\lambda_N) = \frac{1}{N} \sum_{1 \leq j<k \leq N} \log \frac{1}{|\lambda_j - \lambda_k|} + \sum_{j=1}^N V(\lambda_j),
\end{equation}
we define the probability density function
\begin{align}
    \mathcal{P}_{N,s}(\lambda_1, \ldots, \lambda_N) &= \frac{1}{Z_{N,s}} e^{- N \mathsf{H}_N(\lambda_1, \ldots, \lambda_n)} \prod_{j=1}^N \mathbf{1}_{[-s,s]}(\lambda_j),
\end{align}
where $Z_{N,s}$ is the partition function. In particular, we write $ \mathcal{P}_{N} =  \mathcal{P}_{N,\infty} $ for the ensemble without constraint and $Z_N$ for the associated partition function.  
Denoting by $R_k$ and $R_{k,s}$ the $k$-point functions associated to $\mathcal{P}_N$ and $\mathcal{P}_{N,s}$, respectively, we have
\begin{equation} \label{eq. BPS95 lemma 1}
    |R_k(\lambda_1, \ldots, \lambda_k) - R_{k,s}(\lambda_1, \ldots, \lambda_k)| <  R_{k,s}(\lambda_1, \ldots, \lambda_k) e^{-C N}
\end{equation}
for $\lambda_1, \ldots, \lambda_k \in [-s,s]$, where the constant $C>0$ only depends on $\epsilon$ and $s$.
Integrating both handed side of \eqref{eq. BPS95 lemma 1} over $[-s,s]^k$, we then obtain
\begin{equation} \label{eq. BPS95 lemma 1 for probability}
    \bigg| \frac{Z_{N,s}}{Z_N} - 1 \bigg| = \bigg| \int_{[-s,s]^m} \mathcal{P}_N(\lambda_1, \ldots, \lambda_N) d\lambda_1 \cdots d\lambda_N - 1 \bigg| < e^{-C N}.
\end{equation}

Recall that the Hamiltonians $\mathsf{H}_{11}$, $\mathsf{H}_{12}$ and $\mathsf{H}_{22}$ are given by \eqref{def of Ham H11}, \eqref{def of Ham H12} and \eqref{def of Ham H22}, respectively.
We then have the following. 

\begin{lemma} \label{lem. bounds on p_nm: 1}
    For any $n, m, l$, we have
    \begin{equation}
        q_{n,m}^{\rm (low)} \leq p_{n,m} \leq q_{n,m}^{\rm (upp)},
    \end{equation}
    where
    \begin{align}
        q_{n,m}^{\rm (low)} &= \frac{1}{Z_{m,l}} \int_{\HH^l_r} \int_{[-s,s]^m} e^{-n (\mathsf{H}_{11}(\vlambda) + \mathsf{H}_{12}(\vlambda, \vz ) + \mathsf{H}_{22}(\vz))} d\vlambda \, d^2\vz,
        \\
        q_{n,m}^{\rm (upp)} &= \frac{1}{Z_{m,l}} \int_{\HH^l} \int_{[-s,s]^m} e^{-n (\mathsf{H}_{11}(\vlambda) + \mathsf{H}_{12}(\vlambda, \vz^{(r)}) + \mathsf{H}_{22}(\vz))} d\vlambda \, d^2\vz \, ( 1 + e^{-c n})
    \end{align}
    for a sufficiently large constant $s > \sqrt{2(1+\tau)}$, and a constant $c>0$ that only depends on $s$ and $r$.
\end{lemma}
\begin{proof}
    The lower bound is obvious.
    For the upper bound, it follows from
    \begin{equation*}
        \mathsf{H}_{12}(\vlambda, \vz) = \frac{1}{n} \sum_{k=1}^m \sum_{j=1}^l \log \frac{1}{|\lambda_k - z_j|^2}  \, \geq \,  \frac{1}{n} \sum_{k=1}^m \sum_{j=1}^l \log \frac{1}{|\lambda_k - z^{(r)}_j|^2} = \mathsf{H}_{12}(\vlambda, \vz^{(r)})
    \end{equation*}
    that
    \begin{align*}
        p_{n,m} \leq \frac{1}{Z_{m,l}} \int_{\HH^l} \int_{\R^m} e^{-n (\mathsf{H}_{11}(\vlambda) + \mathsf{H}_{12}(\vlambda, \vz^{(r)} ) + \mathsf{H}_{22}(\vz))} d\vlambda \, d^2\vz.
    \end{align*}
    Note that the sum of Hamiltonians $$\mathsf{H}_{11}(\vlambda) + \mathsf{H}_{12}(\vlambda, \vz^{(r)})  = \frac{1}{n} \sum_{1 \leq j < k \leq m} \log \frac{1}{|\lambda_j - \lambda_k|} + \sum_{j=1}^m \Big( \frac{\lambda_j^2}{2(1+\tau)} + \frac{2}{n} \sum_{k=1}^l \log \frac{1}{|\lambda_j - z_k^{(r)}|} \Big)$$ can be interpreted as the Hamiltonian of the eigenvalues of the GOE with finite point charge insertions at $z_1^{(r)}, \ldots, z_l^{(r)}$.
    Especially, the points $z_k^{(r)}$ are bounded away from the real line which provides sufficient regularity to apply \eqref{eq. BPS95 lemma 1 for probability} for sufficiently large constant $s$.
    Therefore, there exists a constant $c>0$ depending only on $s$ and $r$ such that 
    \begin{equation*}
        \int_{\R^m} e^{-n (\mathsf{H}_{11}(\vlambda) + \mathsf{H}_{12}(\vlambda, \vz^{(r)} ) )} d\vlambda \leq \int_{[-s,s]^m} e^{-n (\mathsf{H}_{11}(\vlambda) + \mathsf{H}_{12}(\vlambda, \vz^{(r)} ) )} d\vlambda ( 1 + e^{-c n})
    \end{equation*}
    for any $\vz \in \HH^l$.
    Putting them together, we obtain the desired result for the upper bound.
\end{proof}

We apply the strong Szeg\H{o} theorem for the GOE to illustrate the mean-field behaviour of the real eigenvalues of the elliptic GinOE, see e.g. \cite[Theorem 1]{Shc13}. This theorem was established in \cite{Joh98} in a more general framework, the convergence rate was derived in \cite{Shc13}, and the full asymptotic expansion is provided in \cite{BG13}.

Let $(\lambda_1, \ldots, \lambda_N)$ be the eigenvalues of the GOE of size $N$ normalised in a way that the limiting spectrum, the support of the semicircle law is given by $S = [-a,a] \subset \R$ with $a = \sqrt{2(1+\tau)}$. Consider a real-valued test function $h$ which is supported in an $\epsilon$-neighbourhood of $S$ for some $\epsilon>0$, and has bounded derivatives $\Vert \partial h \Vert_\infty, \Vert \partial^6 h \Vert_\infty \leq N^{1/2} \log N $.
Then, we have
\begin{align}
\begin{split} \label{eq. Thm. Strong Szego, uniform.}
    \log \E(e^{\sum_j h(\lambda_j) - N (h, \rho_{\rm sc}) }) = (h, \nu) + (A_{S} h, h) + O\Big(\frac{\Vert \partial h \Vert_\infty^3 + \Vert \partial^6 h \Vert_\infty^3}{N} \Big).
\end{split}
\end{align}
Here, $(\cdot, \cdot)$ is the inner product in $L^2(\R)$, the operator $A_{S}$ is given by 
\begin{equation}
    A_{S} h (x) = \frac{1}{8\pi^2 \sqrt{a^2 - x^2}} \mathbf{1}_{[-a,a]}(x) \operatorname{p.v.} \int_{-a}^a \frac{\partial h(t) \sqrt{a^2 - t^2}}{x-t} \,dt,
\end{equation}
where p.v. denotes the principal value, and the signed measure $\nu$ is given by 
\begin{equation}
    (h, \nu) = \frac{1}{8} \Big( h(-a) + h(a) \Big) - \frac{1}{4\pi} \int_{-a}^a \frac{h(t)}{\sqrt{a^2 - t^2}} \, dt.
\end{equation}

Recall that $Z_{m,l}$ is given by \eqref{eq. def of widetilde Z_ml}, and $p_{n,n}$ is given by \eqref{eq. p n n closed form}. We write 
\begin{equation}
  Z_{ {\rm GOE}(m)}  = p_{m,m} Z_{m,0} \label{eq. def of Z_GOE(m)}
\end{equation}
for the partition function of the GOE. 
We define 
    \begin{align} \label{eq. def of T}
        T(\vz) &= (h_{\vz}, \nu) + (A_S h_{\vz}, h_{\vz})
    \end{align}
    where
    \begin{equation}
        h_{\vz}(\lambda) = \Big(- \frac{l \lambda^2}{1+\tau} + \sum_{j=1}^l \log |\lambda - z^{(r)}_j|^2 \Big) H_{\frac{s+a}{2}, \frac{s-a}{2}}(\lambda).
    \end{equation}
    Here, $H_{c,\epsilon}(\lambda) = \mathbf{1}_{[-c, c]}*\varphi_{\epsilon}(\lambda)$ for given $c, \epsilon>0$, where $\varphi_\epsilon(x)=C^{-1}\exp(\frac{1}{(x/\epsilon)^2-1}) \mathbf{1}_{\{|x|\leq \epsilon\}}$ is the standard mollifier with normalising constant $C$. 
    We note that $\mathbf{1}_{[-a,a]} \leq H_{\frac{s+a}{2}, \frac{s-a}{2}} \leq \mathbf{1}_{[-s,s]}$.

\begin{lemma} \label{lem. bounds on p_nm: 2}
    Let $l$ be a fixed nonnegative integer, $r>0$ be a constant. Then, there exist $\epsilon>0$ such that
    \begin{equation}  \label{eq. bound of p_nm}
        p_{n,m}^{\rm (low)} (1 - \epsilon n^{-1}) \leq
        p_{n,m} \leq p_{n,m}^{\rm (upp)} (1 + \epsilon n^{-1})
    \end{equation}
    for all $n$, where
    \begin{align}
        p_{n,m}^{\rm (low)} &= \frac{Z_{{\rm GOE}(m)}}{Z_{m,l}} \int_{\HH^l_r} e^{- m \mathsf{H}_{\rm sc}(\vz) -n \mathsf{H}_{22}(\vz) + T(\vz)} d^2\vz ,
        \\
        p_{n,m}^{\rm (upp)} &= \frac{Z_{{\rm GOE}(m)}}{Z_{m,l}} \int_{\HH^l} e^{-m \mathsf{H}_{\rm sc}(\vz^{(r)}) - n \mathsf{H}_{22}(\vz) + T(\vz^{(r)})} d^2\vz.
    \end{align} 
\end{lemma}
\begin{proof}
    We have a bound for the Hamiltonian
    \begin{equation}
        - n (\mathsf{H}_{11}(\vlambda) + \mathsf{H}_{12}(\vlambda, \vz) ) \, \leq \, -\bigg( \sum_{1 \leq j < k \leq m} \log \frac{1}{|\lambda_j - \lambda_k|} + m \sum_{j=1}^m \frac{\lambda_j^2}{2(1+\tau)} \bigg) + \sum_{j=1}^m h_{\vz}(\lambda_j),
    \end{equation}
    where the equality holds for $z \in \HH_r^l$.
    Then, the lemma immediately follows from Lemma~\ref{lem. bounds on p_nm: 1} and \eqref{eq. Thm. Strong Szego, uniform.}.
    The uniform boundedness of the error terms follows from that of $\Vert \partial h_{\vz} \Vert_\infty$ and $\Vert \partial^6 h_{\vz}\Vert_\infty$ in $[-s,s]$.
\end{proof}

\subsection{Analysis on the approximated Hamiltonian} \label{subsec. Analysis on the approximated Hamiltonian}

Building on the previous subsection, we define
\begin{align} \label{eq. def of Q_n^(r)}
    Q^{(r)}_n(z) &= \frac{n}{m} Q_{\tau, n}(z) + Q_{\rm sc}(z^{(r)})
\end{align}
for a constant $r \geq 0$.

\begin{lemma}[\textbf{Global minimum of the approximated potential}] \label{lem. existence of global minimum for Q_n's}
    For a sufficiently small $r\geq 0$, the potential $Q^{(r)}_n(z)$ has a unique global minimum $z_{n}^\star$ in $\HH$ which satisfies
    \begin{equation} \label{eq. def of z star}
        z_{n}^\star = z^\star + O(n^{-1}), \qquad z^\star := i \sqrt{\frac{2 (1-\tau)^2}{3-\tau}},
    \end{equation}
    as $n \to \infty$.
    Moreover, we have
    \begin{equation}
    Q_n^{(r)}(z_{n}^\star) - Q_n^{(r)}(0) = \log \Big(\frac{3-\tau}{1+\tau}\Big) + O(n^{-1}) \label{eq. Q_n(z_n) - Q(0)}
    \end{equation}
    and
    \begin{align}
    \nabla^2 Q_n^{(r)}(z_n) = \begin{pmatrix}
            1 + O(n^{-1}) & 0 \\
            0 & \frac{1}{1-\tau} + O(n^{-1})
        \end{pmatrix}. \label{eq. nabla Q_n(z_n)}
    \end{align}
\end{lemma}
\begin{proof}
  We begin by establishing the uniqueness and existence of the global minimum.  
Throughout the proof, we  write \( z = x + i y \) with \( x, y \in \mathbb{R} \), allowing us to express
    \begin{equation} \label{eq. Q_n(z) expression}
        Q_n^{(0)}(z) = - \frac{1}{n} \log\Big[ 4y^2 \erfc\Big( \sqrt{\frac{2n}{1-\tau^2}} y \Big) \Big] + \frac{x^2-y^2}{1+\tau} - \int_\R \log\Big( (x-t)^2 + y^2 \Big) \, d\rho_{\rm sc}(t) .
    \end{equation}

    We claim that $Q_n^{(0)}(z)$ attains its minimum at $x=0$ for a given $y>0$.
    Since $Q_n^{(0)}(z)$ is an even function in $x = \re z$, it is enough to show $\frac{\partial}{\partial x} Q_n^{(0)}(z) > 0$ for $x>0$.  
    By definition, we have
    \begin{equation*}
        \frac{\partial}{\partial x} Q_n^{(0)}(z) = \frac{2x}{1+\tau} - \int_\R \frac{2(x-t)}{(x-t)^2 + y^2} \, d\rho_{\rm sc}(t).
    \end{equation*}
    Observe that as a function of the variable $x$, the integral in the right-hand side of this equation decreases for $x>\sqrt{2(1+\tau)}$.
    Therefore, it is enough to show $\frac{\partial}{\partial x} Q_n^{(0)}(z) > 0$ for $x \in [0, \sqrt{2(1+\tau)}]$.
    Note also that $\frac{\partial}{\partial x} Q_n^{(0)}(z)$ strictly decreases as $y$ decreases to $0$. Therefore, 
    \begin{align*}
    \begin{split}
        \frac{\partial}{\partial x} Q_n^{(0)}(z)  > 
         \frac{2x}{1+\tau} - \operatorname{p.v.} \int_\R \frac{2}{x-t} \, d\rho_{\rm sc}(t) =0,
    \end{split}
    \end{align*}
    where the equality follows from the variational condition (Euler-Lagrange equation) for the semicircle law.  
    This proves the claim.

    We now show that there exists a global minimum on the positive imaginary axis. 
    It follows from straightforward computations that 
    \begin{equation} \label{eq. partial F(0,y)}
        \frac{\partial}{\partial y} Q_n^{(0)}(iy) = - \frac{2}{ny} + \frac{2 c_n}{n} \frac{ e^{-c_n^2 y^2 } }{ \erfc(c_n y) } - \frac{ 2\sqrt{2(1+\tau) + y^2} }{ 1+\tau }, \qquad c_n := \sqrt{\frac{2n}{1-\tau^2}}. 
    \end{equation}
    We can easily see that $\frac{\partial}{\partial y} Q_n^{(0)}(iy) <0$ for sufficiently small $y>0$.
    Differentiating once more, we have
    \begin{equation*}
        \frac{\partial^2}{\partial y^2} Q_n^{(0)}(iy) = \frac{2}{n y^2} + \frac{ 8 (1 - \sqrt{\pi} c_n y e^{- c_n^2 y^2 } \erfc(c_n y ) ) }{  (1-\tau^2) \pi  e^{2 c_n^2 y^2} \erfc(c_n y )^2 } - \frac{ 2y }{ (1+\tau)\sqrt{2(1+\tau) + y^2} }.
    \end{equation*}
    We use inequalities for the complementary error function \cite[Eqs.(7.8.2), (7.8.5)]{NIST}
    \begin{align*}
        \sqrt{\pi} e^{t^2} \erfc(t) \leq \frac{2}{t+\sqrt{t^2 + 4/\pi}},
        \qquad
        \sqrt{\pi} t e^{t^2} \erfc(t) < \frac{2t^2+2}{2t^2 + 3}, 
    \end{align*}
    which holds for any $t \geq 0$. Substituting $t = c_ny$, we have
    \begin{align*}
    \begin{split}
        \frac{ 8(1 - \sqrt{\pi} t e^{ t^2 } \erfc(t) )}{  (1-\tau^2) \pi  e^{2t^2} \erfc(t)^2 } &> \frac{ 8 }{ (2t^2+3)(1-\tau^2) } \frac{1}{\pi  e^{2t^2} \erfc(t)^2}
       > \frac{2}{1-\tau^2}  \frac{ ( t+\sqrt{t^2 + 4/\pi} )^2 }{ 2t^2+3} .
    \end{split}
    \end{align*}
    Notice also that 
    \begin{equation*}
        \frac{ 2y }{ (1+\tau)\sqrt{2(1+\tau) + y^2} } = \frac{2}{1+\tau} \frac{t}{\sqrt{\frac{4n}{1-\tau}+t^2 }}.
    \end{equation*}
   Combining all of the above, {red}and since $n \geq 1 $ and $ \tau \in [0,1)$, we have 
    \begin{align*}
        \frac{\partial^2}{\partial y^2} Q_n^{(0)}(iy) &> \frac{2}{n} \frac{1}{y^2} + \frac{2}{1-\tau^2}  \frac{ ( t+\sqrt{t^2 + 4/\pi} )^2 }{ 2t^2+3} - \frac{2}{1+\tau} \frac{t}{\sqrt{\frac{4n}{1-\tau}+t^2 }}
        \\
        &>\frac{2}{n} \frac{1}{y^2}  + \frac{2}{1-\tau^2}  \frac{ ( t+\sqrt{t^2 + 4/\pi} )^2 }{ 2t^2+3} - \frac{2}{1+\tau} \frac{t}{\sqrt{4+t^2 }} > \frac{2}{n} \frac{1}{y^2} > 0.
    \end{align*}
    Hence, the function $Q_n^{(0)}(iy)$ is convex for $y>0$.
    Therefore, there exists a global minimum $z_n^\star = i y_n^\star$ of $Q_n^{(0)}(z)$.

    Suppose $c_n y \to \infty$ as $n \to \infty$. Recall that the error function has the power series expansion (see e.g. \cite[Eq.(7.12.1)]{NIST})
    \begin{equation} \label{eq. asymp of erfc}
        \erfc(z) = \frac{e^{-z^2}}{\sqrt{\pi} z} \sum_{j=0}^\infty (-1)^j \frac{(1/2)_{j}}{z^{2j}}, 
    \end{equation}
    which holds for $\arg(z) \leq \frac{3}{4}\pi - \delta$, and the error term has the same sign as the first neglected term if $z$ is real positive.
    Then, we have
    \begin{equation*}
        \frac{\partial}{\partial y} Q_n^{(0)}(iy) = \frac{2 y}{1-\tau^2} - \frac{ 2\sqrt{2(1+\tau) + y^2} }{ 1+\tau } + O\Big(\frac{1}{n y}\Big).
    \end{equation*}
    Since $Q_n^{(0)}(iy) $ is convex, we conclude that
    \begin{equation*}
        y_n^\star = \sqrt{\frac{2(1-\tau)^2}{3-\tau}} + O(n^{-1}).
    \end{equation*}
    Note also that the integrals from \eqref{eq. Q_n(z) expression} can be computed using \cite[Eqs. (4.241.3) and (4.295.27)]{GR14}
    \begin{align*}
        \int_0^1 \sqrt{1-x^2} \log x \, dx &= - \frac{\pi}{8} (2 \log 2 + 1),
        \\
        \int_0^1 \sqrt{1-x^2} \log (1 + a x^2) \, dx &= \frac{\pi}{2} \Big(\log \Big(\frac{1 + \sqrt{1+a}}{2}\Big) + \frac{1}{2} \frac{1-\sqrt{1+a}}{1+\sqrt{1+a}} \Big), \qquad a>0.
    \end{align*}
    Then, \eqref{eq. Q_n(z_n) - Q(0)} and \eqref{eq. nabla Q_n(z_n)} follow from straightforward computations. The proof for $r=0$ remains valid for sufficiently small $r\geq 0$. This completes the proof.
\end{proof}

\begin{figure} 
    \centering
    \begin{subfigure}{0.25\textheight}
        \begin{center}
            \includegraphics[width=\linewidth]{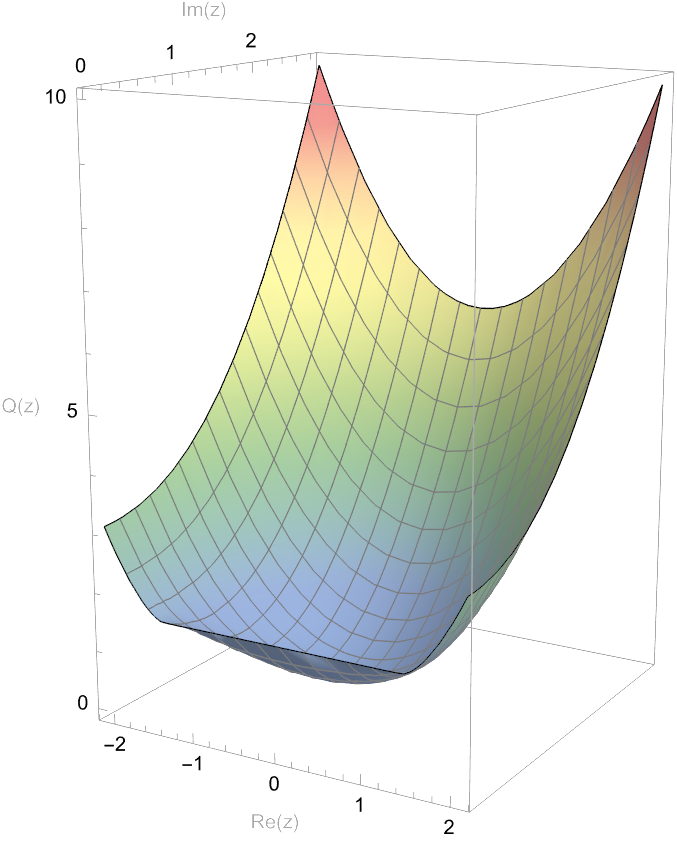}
        \end{center}
        \subcaption{$\tau=0$. }
    \end{subfigure}
    \qquad\qquad\qquad
    \begin{subfigure}{0.25\textheight}
        \begin{center}
            \includegraphics[width=\linewidth]{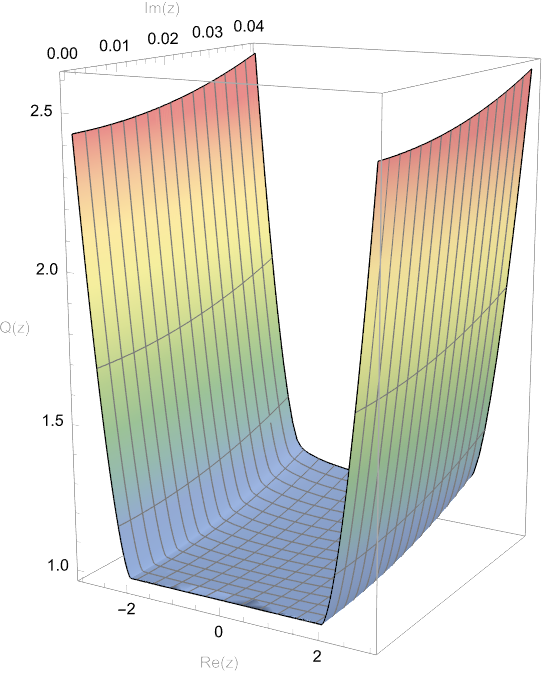} 
        \end{center}
        \subcaption{ $\tau = 1 - 1/n$ }
    \end{subfigure}
    \caption{The plots show the approximated potential $Q_n^{(r)}(z)$ on $z\in\HH$ with $r=0$, $n=100$ for $\tau = 0$ in plot (A) and $\tau = 1-1/n$ in plot (B).}
    \label{fig. 3d plot potential}
\end{figure}

\subsection{Asymptotic analysis around the global minimum}
The main contribution of the integral \eqref{eq. p_nm integral vz} arises from a neighbourhood of the global minimum \( z^\star \).  
This allows us to focus our attention on the neighbourhood.  
We formulate this statement in a more general context.

\begin{lemma}\label{lem. apprx int_H^l to int_K^l}
   Suppose a non-negative function \( p: \mathbb{C}^L \to \mathbb{R} \) has a zero set of measure zero, and a continuous function \( f: \mathbb{C} \to \mathbb{R} \cup \{\infty\} \) attains a unique global minimum at \( z_* \), with \( f(z) \to \infty \) as \( |z| \to \infty \), and
    \begin{equation}
        \int_{\C^L} p(\vz) e^{-\sum_{j=1}^L f(z_j)} d^2\vz < \infty.
    \end{equation}
    Then, for any neighbourhood $K \subset \C^L$ of $(z_*, \ldots, z_*)$, there exists an $\epsilon = \epsilon(K)>0$ such that
    \begin{equation}
        \int_{\C^L} p(\vz) e^{-N \sum_{j=1}^L f(z_j)} d^2 \vz = \int_{K^L} p(\vz) e^{-N \sum_{j=1}^L f(z_j)} d^2 \vz \ \Big(1 + O(e^{-\epsilon N})\Big),
    \end{equation}
     as $N\to\infty$.
\end{lemma}
\begin{proof}
   Without loss of generality, we assume $f(z_*)=0$.
    For a constant $\epsilon>0$, define
    $$
A_k  =  \textstyle \{\vz \in \C^L: k \epsilon \leq \sum_{j=1}^L f(z_j) < (k+1) \epsilon\}, \qquad (k=0,1).
    $$
    Given a neighbourhood $K$ of $z_*$, we choose a sufficiently small $\epsilon>0$, so that $A_0 \cup A_1 \subset K$.
    Since $\sum_j f(z_j) - 2\epsilon \geq 0$ for any $\vz \in \C^L \setminus(A_0 \cup A_1)$, we have 
    \begin{align*}
    \begin{split}
        \int_{\C^L \setminus(A_0 \cup A_1)} p(\vz) e^{-N \sum_j f(z_j)} d^2\vz &= e^{-2 \epsilon N} \int_{\C^L \setminus(A_0 \cup A_1)} p(\vz) e^{-N (\sum_j f(z_j) - 2\epsilon)} d^2\vz
        \\
        &\leq e^{- 2 \epsilon N} \int_{\C^L \setminus(A_0 \cup A_1)} p(\vz) e^{-(\sum_j f(z_j) - 2 \epsilon)} d^2\vz = O(e^{-2 \epsilon N}).
    \end{split}
    \end{align*}
    Similar arguments show that
    \begin{equation*}
        \int_{A_1} p(\vz) e^{-N \sum_j f(z_j)} d^2\vz = O(e^{-\epsilon N}).
    \end{equation*}
    On the other hand, since $\sum_j f(z_j) - \epsilon \leq 0$ for any $\vz \in A_0$, we obtain
    \begin{align*}
    \begin{split}
        \int_{A_0} p(\vz) e^{-N \sum_j f(z_j)} d^2\vz &= e^{-\epsilon N} \int_{A_0} p(\vz) e^{-N (\sum_j f(z_j)-\epsilon)} d^2\vz
        \\
        &\geq e^{- \epsilon N} \int_{A_0} p(\vz) e^{-\sum_j (f(z_j)-\epsilon)} d^2\vz
        \geq C_\epsilon e^{-\epsilon N}
    \end{split}
    \end{align*}
    for some constant $C_\epsilon>0$, which implies the desired result.
\end{proof}

We investigate the asymptotic effect of the Vandermonde determinant on the integral.
While we only need the case $\gamma = 2$ in the following lemma for the proof of Theorem~\ref{Thm. LDP rate function}, we present the general statement for $\gamma > 0$, as the same proof applies without further effort.

\begin{lemma} \label{lem. asymp vandermonde integral}
    Suppose $K \subset \C$ is compact, and a smooth function $f: \C \to \R$ has a unique global minimum $z_*$ in the interior of $K$ with the non-degenerate Hessian $\nabla^2 f(z_*) = 2 \operatorname{diag}(a^2, b^2)$.
    Assume that there exists a function $g(N)$ such that
    \begin{equation}
        \int_K e^{-N f(z)} \, d^2z = g(N) e^{O(1)},
    \end{equation}
    as $N\to\infty$. Then, for any $\gamma >0$ and fixed $L \in \mathbb{N}$, we have
    \begin{equation}
        \int_{K^L}  |\Delta(\vz)|^\gamma e^{-N \sum_{j=1}^L f(z_j)} \,d^2\vz = \frac{g(N)^L}{N^{\gamma L(L-1)/4}} e^{O(1)},
    \end{equation}
    as $N\to\infty$.
\end{lemma}

For the special case when $\gamma=2$, the result immediately follows from the well-known evaluation of the partition function in terms of the orthogonal polynomial norms, see e.g. \cite[Section 5.3]{BF25}. For instance, we have  
\begin{align*}
    \int_{\C^L} |\Delta(\vz)|^2 e^{-N \sum_{j=1}^L |z_j|^2} d^2\vz  = N^{-\frac{L(L+1)}{2}} \pi^L \prod_{j=1}^L j!,
    \qquad 
    \int_{\C^L} e^{-N \sum_{j=1}^L |z_j|^2} d^2\vz  = N^{-L} \pi^L,
\end{align*}
from which it is evident that the absolute square of the Vandermonde determinant contributes a factor of $N^{-L(L-1)/2}$ to the value of the integral, up to a constant.

The following proof is suggested by the anonymous referee whom we greatly appreciate.
An alternative, but longer proof can be found in the first version of arXiv preprint. 
\begin{proof}[Proof of Lemma~\ref{lem. asymp vandermonde integral}]
    We may assume that $z_* = 0$ and $f(z_*)=0$, otherwise consider $f(z+z_*) - f(z_*)$ and $e^{-N f(z_*)}g(N)$ instead of $f(z)$ and $g(N)$, respectively.
    Suppose the assertion is true for $a=b=1$, and consider general $a,b>0$. For a linear map $\mathcal{L}(\vz) = \sqrt{a} \re \vz + i \sqrt{b} \im \vz$, we observe that $\Delta(\mathcal{L}(\vz)) = \Delta(\vz) e^{O(1)}$, as $N\to\infty$. Then, the conclusion directly follows by change of variables.
    Thus, for the rest of the proof, we assume $a=b=1$, which implies $f(z) = |z|^2 + O(|z|^3)$ for sufficiently small $|z|$. 
    
    Following the proof of Lemma \ref{lem. apprx int_H^l to int_K^l}, we have that for any constant $\epsilon \in (0,1/2)$, there exists a constant $c_\epsilon>0$ such that
    \begin{equation*}
        \int_{K^L} |\Delta(\vz)|^\gamma e^{-N \sum_{j=1}^L f(z_j)} d^2 \vz = \int_{B_N^L} |\Delta(\vz)|^\gamma e^{-N \sum_{j=1}^L f(z_j)} d^2\vz \ \Big(1+ O(e^{-c_\epsilon N^{-\epsilon}})\Big),
    \end{equation*}
    as $N\to\infty$, where
    \begin{equation*}
        B_N := \{ z \in K: |z| < N^{-\epsilon} \}.
    \end{equation*}
    Then, there is a constant $c>0$ such that $|z|^2 (1 - c N^{-\epsilon}) \leq f(z) \leq |z|^2 (1 + c N^{-\epsilon})$ for all $z \in B_N$.
    This implies that we have upper and lower bounds
    \begin{equation} \label{eq. proof lem vandermode inequalities}
    \begin{split}
        \mathfrak{B}_{\rm low} \leq \int_{B_N^L} |\Delta(\vz)|^\gamma e^{-N \sum_{j=1}^L f(z_j)} d^2\vz \leq \mathfrak{B}_{\rm upp},
    \end{split}
    \end{equation}
    where
    \begin{align*}
        \mathfrak{B}_{\rm low} := \int_{B_N^L} |\Delta(\vz)|^\gamma e^{- N(1 + c N^{-\epsilon})\sum_{j=1}^L |z_j|^2 } d^2\vz,
        \qquad 
        \mathfrak{B}_{\rm upp} := \int_{B_N^L} |\Delta(\vz)|^\gamma e^{- N(1 - c N^{-\epsilon})\sum_{j=1}^L |z_j|^2 }d^2\vz.
    \end{align*}
    
    By change of variables, we rewrite the upper bound in \eqref{eq. proof lem vandermode inequalities} as
    \begin{equation*}
        \mathfrak{B}_{\rm upp} = \frac{1}{(N(1-cN^{-\epsilon}))^{L+\gamma L(L-1)/4} }\int_{\widetilde{B}_N^L} |\Delta(\mathbf{w})|^\gamma e^{-\sum_{j=1}^L |w_j|^2} d^2\mathbf{w},
    \end{equation*}
    where
    \begin{equation*}
        \widetilde{B}_N := \{ w \in \C: |w| < \sqrt{N(1-c N^{-\epsilon})}N^{-\epsilon}\}.
    \end{equation*}
    Notice that $\widetilde{B}_N \nearrow \C$ as $N\to\infty$.
    Thus, the integral converges to
    \begin{equation*}
        \int_{\widetilde{B}_N^L} |\Delta(\mathbf{w})|^\gamma e^{-\sum_{j=1}^L |w_j|^2} d^2\mathbf{w} \to \int_{\C^L} |\Delta(\mathbf{w})|^\gamma e^{-\sum_{j=1}^L |w_j|^2} d^2\mathbf{w},
    \end{equation*}
    as $N\to\infty$.
    This implies that
    \begin{equation*}
        \mathfrak{B}_{\rm upp} \leq \frac{C_L}{N^{L+\gamma L(L-1)/4} }
    \end{equation*}
    for some constant $C_L>0$.
    Similarly, the lower bound has a bound
    \begin{equation*}
        \frac{c_L}{N^{L+\gamma L(L-1)/4} } \leq \mathfrak{B}_{\rm low}
    \end{equation*}
    for some constant $c_L >0$. Substituting both bounds into \eqref{eq. proof lem vandermode inequalities}, we obtain
    \begin{equation*}
        \int_{K^L} |\Delta(\vz)|^\gamma e^{-N \sum_{j=1}^L f(z_j)} d^2 \vz = \frac{1}{N^{L+\gamma L(L-1)/4} } e^{O(1)},
    \end{equation*}
    as $N\to\infty$.
    Repeating the arguments above, we have
    \begin{equation*}
        \int_K e^{-N f(z)} d^2z = \frac{1}{N} e^{O(1)},
    \end{equation*}
    as $N\to\infty$.
    Putting them together, the conclusion follows.
\end{proof}

\begin{proof}[Proof of Theorem~\ref{Thm. LDP rate function} (i)]
    We show that for a sufficiently small neighbourhood $U$ of $z^\star$, we have
    \begin{equation} \label{eq. asymp of p(n,m)}
        p_{n,m} = \frac{Z_{{\rm GOE}(m)}}{Z_{m,l}} \int_{U^l} |\Delta(\vz)|^2 e^{-m \sum_{j=1}^l (Q_{\rm sc}(z_j) + Q_\tau(z_j)) } d^2\vz \, e^{O(1)}
    \end{equation}
    as $n\to \infty$, where
    \begin{equation}
        Q_\tau(z) = \frac{(\re z)^2}{1+\tau} + \frac{(\im z)^2}{1-\tau}.
    \end{equation}
    For this, we analyse the bounds $p_{n,m}^{\rm (low)}$ and $p_{n,m}^{\rm (upp)}$ in \eqref{eq. bound of p_nm}.
    
    Note that by \eqref{eq. asymp of erfc}, we have 
    \begin{align*}
        Q_{\tau,n}(z) &= Q_\tau(z) + O(n^{-1}),
    \end{align*}
    where the error term is uniformly bounded for $\im z>r$ for any $r >0$. Moreover, $Q_{\tau, n}(z) \geq Q_\tau(z) $ for all $z \in \HH$.
    Therefore in \eqref{eq. bound of p_nm}, for the lower bound , we have
    \begin{align*}
        p_{n,m}^{\rm (low)} = \frac{Z_{{\rm GOE}(m)}}{Z_{m,l}} \int_{\HH_r^l}  e^{-m \sum_{j=1}^l (Q_{\rm sc}(z_j) + Q_\tau(z_j)) + T(\vz) + O(1) } \prod_{1 \leq j<k \leq l} |z_j-z_k|^2 |z_j - \overline{z}_k|^2 d^2\vz,
    \end{align*}
    and for the upper bound, we have
    \begin{align*}
        p_{n,m}^{\rm (upp)} - p_{n,m}^{\rm (low)} &=  \frac{Z_{{\rm GOE}(m)}}{Z_{m,l}} \int_{\HH^l \setminus \HH_r^l}  e^{-m \sum_{j=1}^l (Q_{\rm sc}(z_j) + Q_\tau(z_j)) + T(\vz^{(r)}) + O(1) } \prod_{1 \leq j<k \leq l} |z_j-z_k|^2 |z_j - \overline{z}_k|^2 d^2\vz,
    \end{align*}
    with the error terms in the integrals being uniformly bounded.
    Applying Lemma \ref{lem. apprx int_H^l to int_K^l}, we can restrict our integration to the $\epsilon$-neighbourhood $U$ of $z^\star$ by introducing an $O(e^{-\delta n})$ error. 
    For sufficiently small $\epsilon>0$ and $\vz \in U^l$, we have
    \begin{align*}
        T(\vz) = T(\vz^{(r)}) = O(1), \qquad \prod_{1 \leq j<k \leq l} |z_j - \overline{z}_k|^2 = O(1),
    \end{align*}
    because $h_{\vz}$ in \eqref{eq. def of T} is bounded and
    \begin{equation*}
        (2z^\star - 2\epsilon)^{\frac{l(l-1)}{2}} \leq \prod_{1 \leq j<k \leq l} |z_j - \overline{z}_k|^2 \leq (2z^\star + 2\epsilon)^{\frac{l(l-1)}{2}}.
    \end{equation*}
    Therefore, we obtain \eqref{eq. asymp of p(n,m)}.
    
    Putting $m=n-2$ and $l=1$ in \eqref{eq. asymp of p(n,m)}, we obtain
    \begin{align*}
        p_{n,n-2} &= \frac{Z_{{\rm GOE}(n-2)}}{Z_{n-2,1}} g(n-2),
    \end{align*}
    where
    \begin{equation*}
        g(n) = \int_{U}  e^{-n (Q_{\rm sc}(z) + Q_\tau(z))} d^2z \, e^{O(1)}.
    \end{equation*}
    By Lemma \ref{lem. asymp vandermonde integral}, for general $m = n- 2l$ with constant $l$, we have
    \begin{align*}
        p_{n,n-2l} &= \frac{Z_{{\rm GOE}(m)}}{Z_{m,l}} \int_{U^l}  |\Delta(\vz)|^2 e^{-m \sum_{j=1}^l (Q_{\rm sc}(z_j) + Q_\tau(z_j))} d^2\vz
        \\
        &= \frac{Z_{{\rm GOE}(m)}}{Z_{m,l}} \frac{g(m)^l}{n^{l(l-1)/2}} e^{O(1)} = \frac{Z_{{\rm GOE}(m)}}{Z_{m,l}} \frac{1}{n^{l(l-1)/2}} \Big( \frac{Z_{m,1} \,  p_{m+2,m} }{Z_{{\rm GOE}(m)}} \Big)^l \, e^{O(1)}.
    \end{align*}
    Substituting \eqref{eq. LDP sH l=1}, \eqref{eq. def of widetilde Z_ml}, and \eqref{eq. def of Z_GOE(m)}, we obtain the desired result.
\end{proof}

\section{General case: weak non-Hermiticity} \label{Sec. l finite wH}

In this section, we examine the asymptotic behaviour of the probability \( p_{n,m} \) for \( \tau = 1 - \alpha^2 / n \). Using the Pfaffian integral representation \eqref{eq. p(n,m) pfaffian integral representation} of \( p_{n,m} \), we establish that the dominant contributions to the integral originate from regions close to the real axis. The analysis proceeds as follows.

\begin{itemize}
    \item We begin by exploring an alternative representation of the prekernel \( \kappa_{n}(\zeta, \eta) \) via a summation formula. This representation is pivotal for our subsequent calculations and is presented in Lemma~\ref{lem. Alt rep of prekernel}.
    \smallskip
    \item Next, we examine the uniform asymptotic behaviour of the Hermite functions, as detailed in Lemma~\ref{lem. P-R for the Hermite function}, to derive asymptotic formulas for the prekernel in various regions. Specifically, we analyse the behaviour of the prekernel near \(  \mathbb{R} \subset \mathbb{C} \) in Lemmas~\ref{lem. asymp prekernel kappa(z,z)} and \ref{lem. asymp prekernel kappa(z,w)}, and in the exterior region in Lemma~\ref{lem. asymp prekernel kappa(z,w) far away}.
    \smallskip    
    \item Finally, employing the asymptotics of the prekernel, we prove Theorem~\ref{Thm. LDP rate function} for the weakly non-Hermitian regime.
\end{itemize}

\subsection{Alternative representations of the prekernel}

To analyse the prekernel \eqref{eq. def of prekernel kappa}, it is convenient to use the Hermite function
\begin{equation}
    \psi_N(z) = \frac{1}{2^N c_N} H_{N}(z) e^{-\frac{z^2}{2}}, \qquad c_N = \pi^\frac{1}{4} 2^{-\frac{N}{2}} \sqrt{N!}. 
\end{equation} 

\begin{lemma}[\textbf{Alternative representations of the GOE prekernel}] \label{lem. Alt rep of prekernel}
    For a positive even integer $n$, we have
    \begin{align}
    \begin{split} \label{eq. rational rep prekernel}
        \kappa_{n}(\zeta, \eta) = - \frac{\sqrt{n}}{2\sqrt{2}} &\Bigg[ \frac{\Psi_{n,1}(\zeta,\eta)}{\zeta-\eta} - \frac{\Psi_{n,2}(\zeta,\eta)}{(\zeta-\eta)^2} \Bigg],
    \end{split}
    \end{align}
    where 
    \begin{align}
    \begin{split}
        \Psi_{n,1}(\zeta,\eta) &= \sqrt{2n} \psi_{n-1}(\zeta) \psi_{n-1}(\eta) - \sqrt{2(n-1)} \psi_{n-2}(\zeta) \psi_{n}(\eta) - \eta \psi_n(\zeta) \psi_{n-1}(\eta) + \zeta \psi_{n-1}(\zeta) \psi_n(\eta),
        \\
        \Psi_{n,2}(\zeta,\eta) &= \psi_n(\zeta) \psi_{n-1}(\eta) - \psi_{n-1}(\zeta) \psi_n(\eta).
    \end{split}
    \end{align}
    Equivalently, we have
    \begin{equation} \label{eq. integral rep prekernel}
        \kappa_{n}(\zeta,\eta) = - \frac{\sqrt{n}}{2\sqrt{2}} \Bigg[ \psi_{n-1}(\eta) \int_0^1 t \, \widetilde{\psi}_n(s_t) \, dt - \psi_n(\eta) \int_0^1 t \, \widetilde{\psi}_{n-1}(s_t) \, dt  + \psi_n(\zeta) \psi_{n-1}(\eta) \Bigg],
    \end{equation}
    where $s_t = t \zeta + (1-t) \eta$ and
    \begin{align}\label{eq. def widetilde psi_(n,1)}
        \widetilde{\psi}_n(s) &= (s^2 - 1) \psi_n(s) - 2 \sqrt{2n} s \psi_{n-1}(s) + 2 \sqrt{n(n-1)} \psi_{n-2}(s). 
    \end{align}
\end{lemma}
\begin{proof}
Recall that the GOE skew-orthogonal polynomials $q_j(\zeta)$ are given in \eqref{eq. skew-ortho polys}. Then, it is well known that 
\begin{align} \label{eq. sum formula for GOE skew OP}
\begin{split}
    &\sum_{j=0}^{n/2-1} \frac{q_{2j+1}(\zeta) q_{2j}(\eta) - q_{2j}(\zeta) q_{2j+1}(\eta)}{h_j} 
    \\
    & = - e^{\frac{\zeta^2 + \eta^2}{2}} \frac{c_{n}}{c_{n-1}} \Bigg[ \frac{d}{d\zeta}\Big(\frac{\psi_{n}(\zeta) \psi_{n-1}(\eta) - \psi_{n-1}(\zeta) \psi_{n}(\eta)}{\zeta-\eta} \Big) + \psi_{n}(\zeta) \psi_{n-1}(\eta)\Bigg],
\end{split}
\end{align}
see e.g. \cite[Eq.(5.2.1)]{BS06} and references therein.
Then, the rational expression \eqref{eq. rational rep prekernel} is a direct consequence of \eqref{eq. def of prekernel kappa}, \eqref{eq. sum formula for GOE skew OP}, and the recurrence relation
    \begin{align} \label{eq. derivative of the Hermite func.}
        \psi_N'(\zeta) = \sqrt{2N} \psi_{N-1}(\zeta) - \zeta \psi_N(\zeta)
    \end{align}
    for any positive $N$. 
    
To show \eqref{eq. integral rep prekernel}, we rewrite
    \begin{align*}
        \frac{\psi_n(\zeta) \psi_{n-1}(\eta) - \psi_{n-1}(\zeta) \psi_n(\eta)}{\zeta-\eta} &= \frac{\psi_n(\zeta) - \psi_n(\eta)}{\zeta - \eta} \psi_{n-1}(\eta) - \frac{\psi_{n-1}(\zeta) - \psi_{n-1}(\eta)}{\zeta - \eta} \psi_n(\eta)
        \\
        &= \psi_{n-1}(\eta) \int_0^1 \psi_n'(t \zeta + (1-t) \eta) \,dt - \psi_n(\eta) \int_0^1 \psi_{n-1}'(t \zeta + (1-t) \eta) \,dt,
    \end{align*}
    which leads to 
    \begin{align*}
    &\quad \frac{d}{d\zeta}\Big( \frac{\psi_n(\zeta) \psi_{n-1}(\eta) - \psi_{n-1}(\zeta) \psi_n(\eta)}{\zeta-\eta} \Big) 
     \\
     & = \psi_{n-1}(\eta) \int_0^1 t \psi_n''(t \zeta + (1-t) \eta) \,dt - \psi_n(\eta) \int_0^1 t \psi_{n-1}''(t \zeta + (1-t) \eta) \,dt.
    \end{align*} 
    Then, the conclusion follows from \eqref{eq. derivative of the Hermite func.} by writing
    \begin{align*}
    \begin{split}
        \psi_N''(\zeta) &= \sqrt{2N} \psi_{N-1}'(\zeta) - \psi_N(\zeta) - \zeta \psi_N'(\zeta)
        \\
        &= (\zeta^2 - 1) \psi_N(\zeta) - 2 \sqrt{2N} \zeta \psi_{N-1}(\zeta) + 2 \sqrt{N (N-1)} \psi_{N-2}(\zeta) = \widetilde{\psi}_N(\zeta).
    \end{split}
    \end{align*}
\end{proof}

\subsection{Plancherel-Rotach formula in the complex plane}

The asymptotic behaviour of orthogonal polynomials in the complex plane has been extensively studied in the literature, including the seminal work \cite{DKMVZ99}. This behaviour depends on the location within the complex plane.
To illustrate this, we write 
\begin{equation}
  {\rm I}  = \{ z \in \C: |\re z| \leq 1, |\im z| < \delta, |z - 1| \geq \delta/\sqrt{2}, |z + 1| \geq \delta \},
\end{equation}
and 
\begin{equation}
    {\rm II}^{\pm}  = \{ z \in \C: |z \mp 1| \leq \delta \},
    \qquad 
    {\rm III} = \C \setminus ({\rm I} \cup {\rm II}^+ \cup {\rm II}^-),  
\end{equation}
for some sufficiently small $\delta>0$, see Fig. \ref{fig:enter-label}. 

\begin{figure}
    \centering
        \begin{tikzpicture}
            \def \tr {-1}
            \def \sc {1.25}
            \def \he {0.5} % height of the figure below
            \begin{scope}[scale = 3.5]
                \draw[thick,->] (-1.5,0) -- (1.5, 0);
                \fill[nearly transparent, gray] (-1.5,0) rectangle (1.5,0.1);
                \draw[dotted] (-1.5,0) rectangle (1.5,0.1);
                \draw[|<->|] (-1.4,0) -- (-1.4,0.1) node[above] {$n^{\epsilon_i-1}$};

                \fill[nearly transparent, red] (-1.05,0) rectangle (-0.95,0.1);
                \fill[nearly transparent, red] (1.05,0) rectangle (0.95,0.1);

                \begin{scope}
                    \clip (-1.5,0) rectangle (1.5,0.7);
                    \draw (1,0) circle (0.3);
                    \clip (-1.5,0) rectangle (1.5,0.7);
                    \draw (-1,0) circle (0.3);
                \end{scope}
                \draw[black] ({-1+0.3*sqrt(2)/2},{0.3*sqrt(2)/2}) -- ({1-0.3*sqrt(2)/2}, {0.3*sqrt(2)/2});

                \fill[black] (-1, 0) circle (.5pt) node[below] {-1};
                \fill[black] (1, 0) circle (.5pt) node[below] {1};
                
                \node at (0, 0.15) {${\rm I}$};
                \node at ({1+0.03},{0.15+0.01}) {${\rm II}^+$};
                \node at ({-1+0.03},{0.15+0.01}) {${\rm II}^-$};
                \node at (0, 0.5) {${\rm III}$};

                \begin{scope}
                    \draw[dashed] (-1.5,0) -- ({-\sc*1.5}, {0 + \tr});
                    \draw[dashed] (1.5,0) -- ({\sc*1.5}, {0 + \tr});
                    \draw[dashed] (1.5,0.1) -- ({\sc*1.5}, {\sc*\he + \tr});
                    \draw[dashed] (-1.5,0.1) -- ({-\sc*1.5}, {\sc*\he + \tr});
                \end{scope}
                \begin{scope}[shift={(0,\tr)}, scale=\sc]
                    \draw[thick,->] (-1.5,0) -- (1.5,0);
                    
                    \fill[nearly transparent, gray] (-1.5,0) rectangle (1.5,\he);
                    \draw[dotted, thick] (-1.5,0) rectangle (1.5,\he);
                    
                    \draw[dashed, red] (1,0) -- (1,{\he/2-0.05});
                    \draw[dashed, red] (1,{\he/2+0.05}) -- (1,\he);
                    \draw[dashed, red] (-1,0) -- (-1,{\he/2-0.05});
                    \draw[dashed, red] (-1,{\he/2+0.05}) -- (-1,\he);
                    \fill[red,opacity=0.2] (-1.15,\he) rectangle (-0.85,0);
                    \fill[red,opacity=0.2] (1.15,\he) rectangle (0.85,0);
                    \draw[|<->|] (-1,0.05) -- (-0.85,0.05) node[above] {$n^{\epsilon_c-\frac{2}{3}}$};
                    
                    \fill[black] (-1, 0) circle (.5pt) node[below] {-1};
                    \fill[black] (1, 0) circle (.5pt) node[below] {1};
            
                    \node at (0, {\he/2}) {${\rm I}_n$};
                    \node[red] at (-1, {\he/2}) {${\rm II}_n^-$};
                    \node[red] at (1, {\he/2}) {${\rm II}_n^+$};
                    \node at (-1.3, {\he/2}) {${\rm III}_n$};
                    \node at (1.3, {\he/2}) {${\rm III}_n$};
                \end{scope}
            \end{scope}
        \end{tikzpicture}
    \caption{Illustration of the regions $\rm I, II^{\pm}, III$ and ${\rm I}_n, {\rm II}_n^\pm, {\rm III}_n$  in the upper half plane.}
    \label{fig:enter-label}
\end{figure}
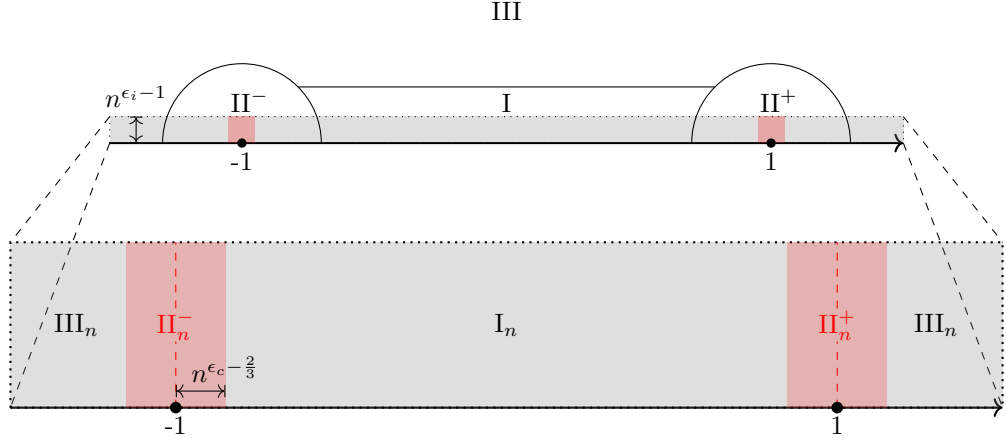

The asymptotic behaviour of the Hermite polynomials can be derived as a special case of \cite[Theorem 2.2]{DKMVZ99}. To present these results, we introduce some notation adapted to our purpose.
Let us define functions $\phi^\pm: \C_{\rm cut}=\C \setminus ((-\infty,-1] \cup [1, \infty)) \to \C$ by
\begin{equation}
    \phi^\pm(z) = 2 \int_{\pm 1}^z \sqrt{1-w^2} \, dw,
\end{equation}
where the integrand is defined on $\C_{\rm cut}$ as the analytic extension of the semicircle density on $(-1,1)$ to $\C_{\rm cut}$.
We also define $f_N^{\pm}$ in small neighbourhoods of $z = \pm 1$ in $\C_{\rm cut}$, respectively, satisfying
\begin{align} \label{eq. def of f_n}
    (- f_N^+(z))^{\frac{3}{2}} = - N \frac{3}{2} \phi^+(z),
    \qquad f_N^-(z)^{\frac{3}{2}} = N \frac{3}{2} \phi^-(z).
\end{align}
Then, there exist analytic functions $\hat{f}^\pm(z)$ in small neighbourhoods of $\pm1$ in $\C$, respectively, such that $\hat{f}^+(1) = \hat{f}^-(-1) = 1$ and
\begin{align} \label{eq. def of f hat}
    f_N^+(z) = 2 N^{\frac{2}{3}} (z-1) \hat{f}^+(z), \qquad   f_N^-(z) = 2 N^{\frac{2}{3}} (z+1) \hat{f}^-(z),
\end{align}
see \cite[Eq.(4.10)]{DG07}. For simplicity, we shall often omit the superscripts $+$ and write $\phi$, $f_N$, and $\hat{f}$ in place of $\phi^+$, $f_N^+$, and $\hat{f}^+$, respectively.
For an integer $M$, we denote
\begin{equation}
    \theta_M(z) = \Big(M-\frac{1}{2}\Big) \arccos(z) - \frac{\pi}{4}.
\end{equation}
Note that $\phi(z)$ is also expressed as
\begin{align}
    \phi(z) = z \sqrt{1-z^2} - \arccos(z),
    \qquad 
    i \phi(z) = z \sqrt{z^2-1} - \log(z + \sqrt{z^2-1}).
\end{align}
Recall also that the Airy function is defined by 
\begin{equation}
    \Airy(x)= \frac{1}{\pi} \int_0^\infty \cos\Big( \frac{t^3}{3}+xt \Big)\,dt
\end{equation}
on the real line, and by analytic continuation elsewhere.

Using the notation above, it follows from \cite[Theorem 2.2]{DKMVZ99} that there exists a sufficiently small \(\delta_0 > 0\) such that, for any \(\delta \in (0, \delta_0]\), the following holds.
\begin{itemize}
    \item For $z \in {\rm I}$, we have
\begin{align}
\begin{split} \label{eq. DKMVZ99 Plancherel-Rotach osc}
    \psi_{N}(\sqrt{2N} z)& = 2^{\frac{1}{4}} \pi^{-\frac{1}{2}} N^{-\frac{1}{4}} (1-z^2)^{-\frac{1}{4}} 
    \\
    &\quad \times \bigg[ \cos(N\phi(z) - \theta_0(z)) \Big(1 + O(N^{-1})\Big) + \sin(N\phi(z) - \theta_0(z)) O(N^{-1}) \bigg].
\end{split}
\end{align}
\item For $z \in {\rm II}^+$, we have
\begin{align}
\begin{split} \label{eq. DKMVZ99 Plancherel-Rotach airy}
    \psi_{N}(\sqrt{2N} z) = 2^{-\frac{1}{4}} N^{-\frac{1}{4}} &\bigg[ \frac{(z+1)^{\frac{1}{4}}}{(z-1)^{\frac{1}{4}}} f_N(z)^{\frac{1}{4}} \Airy(f_N(z)) \Big( 1 + O(N^{-1}) \Big)
    \\
    &\quad - \frac{(z-1)^{\frac{1}{4}}}{(z+1)^{\frac{1}{4}}} f_N(z)^{-\frac{1}{4}} \Airy'(f_N(z)) \Big( 1 + O(N^{-1}) \Big) \bigg]. 
\end{split}
\end{align}
An asymptotic result analogous to \eqref{eq. DKMVZ99 Plancherel-Rotach airy} holds for \( z \in \mathrm{II}^- \).
\item For $z \in {\rm III}$, we have
\begin{equation} \label{eq. DKMVZ99 Plancherel-Rotach exp}
    \psi_{N}(\sqrt{2N}z) = 2^{-\frac{3}{4}} \pi^{-\frac{1}{2}} N^{-\frac{1}{4}} e^{- N i \phi(z)} \bigg( \frac{(z-1)^\frac{1}{4}}{(z+1)^\frac{1}{4}} + \frac{(z+1)^\frac{1}{4}}{(z-1)^\frac{1}{4}} \bigg) \Big( 1 + O(N^{-1}) \Big).
\end{equation}
\end{itemize}
In the above formulas, all error terms are uniform within the specified region.

For our purposes, we require not only the asymptotics of \(\psi_{N}(\sqrt{2N}z)\), but also those of \(\psi_{N+M}(\sqrt{2N}z)\) for \(M = 0, -1, -2\).  
To simplify and adapt them for our needs, we write 
\begin{align}
    \Omega_n &= \{ z \in \C: |\im z| < n^{\epsilon_i - 1} \},
    \\
    {\rm I}_n &= \{ z \in \Omega_n: |\re z| \leq 1 - \delta n^{\epsilon_c-\frac{2}{3}}\},
    \\
    {\rm II}_n^{\pm} &= \{ z \in \Omega_n: |\re z \mp 1| \leq \delta n^{\epsilon_c-\frac{2}{3}}\},
    \\
    {\rm III}_n &= \{ z \in \Omega_n: |\re z | \geq 1 + \delta n^{\epsilon_c-\frac{2}{3}}\}
\end{align}
for some constants $0<\epsilon_i < \epsilon_c < \tfrac{2}{3}$ and $\delta>0$.

\begin{lemma}[\textbf{Plancherel-Rotach formula for the Hermite function}] \label{lem. P-R for the Hermite function}
    Let $m$ be an integer. Then, the following asymptotic formulas hold.
    \begin{enumerate}
        \item For $z\in {\rm I}_n$, 
        \begin{align}
        \begin{split} \label{eq. Plancherel-Rotach osc near R}
            \psi_{n+m}(\sqrt{2n} z) = 2^\frac{1}{4} \pi^{-\frac{1}{2}} n^{-\frac{1}{4}} (1-z^2)^{-\frac{1}{4}} \bigg[ &\cos(n \phi(z) - \theta_m(z) ) \Big( 1 + O( n^{-1}(1-z^2)^{-\frac{3}{2}} ) \Big)
            \\
            & \quad + \sin(n \phi(z) - \theta_m(z) ) \, O( n^{-1}(1-z^2)^{-\frac{3}{2}} ) \bigg];
        \end{split}
        \end{align}
        \item For $z\in {\rm II}_n^+$,  
        \begin{align}
        \begin{split} \label{eq. Plancherel-Rotach Airy near R}
            \psi_{n+m}(\sqrt{2n} z) = 2^\frac{1}{4} n^{-\frac{1}{12}} \bigg[ & \Airy(2 n^\frac{2}{3}(z-1)) \Big(1 + O(z-1) + O(n^{-1}) \Big) 
            \\
            & \quad - (m+\tfrac{1}{2}) \Airy'(2n^\frac{2}{3}(z-1)) n^{-\frac{1}{3}} \Big( 1 + O(z-1) + O(n^{-1}) \Big) \bigg]. 
        \end{split}
        \end{align}
        \item For $z\in {\rm III}_n$,
        \begin{align}
        \begin{split} \label{eq. Plancherel-Rotach exp near R}
            \psi_{n+m}(\sqrt{2n} z) &= 2^{-\frac{3}{4}} \pi^{-\frac{1}{2}} n^{-\frac{1}{4}} e^{-(n+m) i \phi(z) + m i z \phi'(z)/2}
            \\
            &\quad \times \bigg( \frac{(z-1)^\frac{1}{4}}{(z+1)^\frac{1}{4}} + \frac{(z+1)^\frac{1}{4}}{(z-1)^\frac{1}{4}} \bigg) \Big(1 + O(n^{-1}) + O(n^{-1}(z^2-1)^{-\frac{3}{2}})\Big). 
        \end{split}
        \end{align}
    \end{enumerate}
    Here, all error terms are uniformly bounded.
\end{lemma}
\begin{proof}
  Let us write 
  $$\psi_{n+m}(\sqrt{2(n+m)}\widetilde{z}) = \psi_{n+m}(\sqrt{2n}z), \qquad \widetilde{z} = \sqrt{ \frac{n}{n+m} } z. $$
The lemma then follows from straightforward computations using the asymptotics \eqref{eq. DKMVZ99 Plancherel-Rotach osc}–\eqref{eq. DKMVZ99 Plancherel-Rotach exp}. For the reader's convenience, we provide additional details.

    Suppose $\widetilde{z}\in {\rm I}$. By applying \eqref{eq. DKMVZ99 Plancherel-Rotach osc}, we have 
    \begin{align*}
    \begin{split}
        \psi_{n+m}(\sqrt{2n} z)& = 2^{\frac{1}{4}} \pi^{-\frac{1}{2}} (n+m)^{-\frac{1}{4}} (1-\widetilde{z}^2)^{-\frac{1}{4}} 
        \\
        &\quad \times \bigg[ \cos((n+m)\phi(\widetilde{z}) - \theta_0(\widetilde{z})) \Big(1 + O(n^{-1})\Big)  + \sin((n+m)\phi(\widetilde{z}) - \theta_0(\widetilde{z})) O(n^{-1}) \bigg].
    \end{split}
    \end{align*}
    For each term, we have uniform approximations
    \begin{align*}
        (1-\widetilde{z}^2)^{-\frac{1}{4}} &= (1-z^2)^{-\frac{1}{4}} (1 + O(n^{-1}(1-z^2)^{-1})),
        \\
        (n+m)\phi(\widetilde{z}) - \theta_0(\widetilde{z}) &= n\phi(z) - \theta_m(z) + O(n^{-1}(1-z^2)^{-\frac{1}{2}}).
    \end{align*}
    Substituting back in the above, we obtain \eqref{eq. Plancherel-Rotach osc near R}.
    
    Suppose $\widetilde{z} \in {\rm I}_n \cap {\rm II}$. 
    It follows from \eqref{eq. DKMVZ99 Plancherel-Rotach airy} that
    \begin{align}
    \begin{split} \label{eq. DKMVZ99 Plancherel-Rotach (n+m)-case airy}
        \psi_{n+m}(\sqrt{2n} z) = 2^{-\frac{1}{4}} (n+m)^{-\frac{1}{4}} &\bigg( \frac{(\widetilde{z}+1)^{\frac{1}{4}}}{(\widetilde{z}-1)^{\frac{1}{4}}} f_{n+m}(\widetilde{z})^{\frac{1}{4}} \Airy(f_{n+m}(\widetilde{z})) \Big( 1 + O(n^{-1}) \Big)
        \\
        &\quad - \frac{(\widetilde{z}-1)^{\frac{1}{4}}}{(\widetilde{z}+1)^{\frac{1}{4}}} f_{n+m}(\widetilde{z})^{-\frac{1}{4}} \Airy'(f_{n+m}(\widetilde{z})) \Big( 1 + O(n^{-1}) \Big) \bigg). 
    \end{split}
    \end{align}
    We use asymptotic of the Airy function \cite[Eqs.(9.7.9), (9.7.10) and \S9.7(iv)]{NIST}
    \begin{align} \label{eq. asymp Airy(-z)}
    \begin{split}
        \Airy(-t) &= \pi^{-\frac{1}{2}} t^{-\frac{1}{4}} \Big[ \cos\Big( s - \frac{\pi}{4} \Big) ( 1 + O(s^{-1}) ) + \sin\Big( s - \frac{\pi}{4} \Big) O(s^{-1}) \Big],
        \\
        \Airy'(-t) &= \pi^{-\frac{1}{2}} t^{\frac{1}{4}} \Big[ \cos\Big( s - \frac{\pi}{4} \Big) O(s^{-1}) + \sin\Big( s - \frac{\pi}{4} \Big) ( 1 + O(s^{-1}) ) \Big],
    \end{split}
    \end{align}
    where $s = \frac{2}{3} t^{\frac{3}{2}}$ and $|\arg t| \leq \frac{2}{3}\pi - \epsilon$ for any $\epsilon>0$, which gives
    \begin{align*}
    \begin{split}
        \psi_{n+m}(\sqrt{2n} z) = 2^{-\frac{1}{4}} (n+m)^{-\frac{1}{4}} &\bigg( \frac{(\widetilde{z}+1)^{\frac{1}{4}}}{(\widetilde{z}-1)^{\frac{1}{4}}} e^{i\frac{\pi}{4}} \cos(-(n+m)\phi(\widetilde{z})+\frac{\pi}{4}) \Big( 1 + O(n^{-1}\phi(\widetilde{z})^{-1}) \Big)
        \\
        &\quad - \frac{(\widetilde{z}-1)^{\frac{1}{4}}}{(\widetilde{z}+1)^{\frac{1}{4}}} e^{i\frac{\pi}{4}} \sin(-(n+m)\phi(\widetilde{z})+\frac{\pi}{4}) \Big( 1 + O(n^{-1}\phi(\widetilde{z})^{-1}) \Big) \bigg).
    \end{split}
    \end{align*}
    Then, the asymptotic \eqref{eq. Plancherel-Rotach osc near R} readily follows from the Taylor expansion and trigonometric identities.

    Suppose $\widetilde{z} \in {\rm II}_n$.
    By \eqref{eq. def of f hat}, we have 
    \begin{align*}
        \frac{(\widetilde{z}+1)^{\frac{1}{4}}}{(\widetilde{z}-1)^{\frac{1}{4}}} f_{n+m}(\widetilde{z})^{\frac{1}{4}} &= 2^{\frac{1}{2}} (n+m)^{\frac{1}{6}} (1 + O(\widetilde{z}-1)),
        \\
        f_{n+m}(\widetilde{z}) &= 2 (n+m)^{\frac{2}{3}} (w-1) (1 + O(w-1))
        \\
        &= \Big( 2 n^{\frac{2}{3}} (z-1) - m n^{-\frac{1}{3}} \Big) (1 + O(n^{-1}) + O(w-1)).
    \end{align*}
    Using the Taylor expansion and \eqref{eq. DKMVZ99 Plancherel-Rotach (n+m)-case airy}, the desired asymptotic behaviour \eqref{eq. Plancherel-Rotach Airy near R} follows.

    Suppose $\widetilde{z} \in {\rm III}_n \cap {\rm II}$.
    For \eqref{eq. DKMVZ99 Plancherel-Rotach (n+m)-case airy}, we use the asymptotics of the Airy functions \cite[Eqs.(9.7.5), (9.7.6)]{NIST}
    \begin{align} \label{eq. asymp Airy(z)}
    \begin{split}
        \Airy(t)  = 2^{-1} \pi^{-\frac{1}{2}} z^{-\frac{1}{4}} e^{-s} \Big( 1 + O(s^{-1}) \Big), \qquad 
     \Airy'(t) = -2^{-1} \pi^{-\frac{1}{2}} z^{\frac{1}{4}} e^{-s} \Big( 1 + O(s^{-1}) \Big),
    \end{split}
    \end{align}
    where $s = \frac{2}{3} t^{\frac{3}{2}}$ and $|\arg t| \leq \pi - \epsilon$ for any $\epsilon>0$.
    Together with \eqref{eq. def of f_n}, we obtain
    \begin{align*}
        \psi_{n+m}(\sqrt{2n} z) = 2^{-\frac{5}{4}} \pi^{-\frac{1}{2}} n^{-\frac{1}{4}} e^{-(n+m) i \phi(\widetilde{z})} \bigg( \frac{(\widetilde{z}-1)^\frac{1}{4}}{(\widetilde{z}+1)^\frac{1}{4}} + \frac{(\widetilde{z}+1)^\frac{1}{4}}{(\widetilde{z}-1)^\frac{1}{4}} \bigg) \Big( 1 + O(n^{-1} \phi(\widetilde{z})^{-1}) \Big).
    \end{align*}
    Here, the series expansion gives 
    \begin{equation*}
        \phi(\widetilde{z}) = \phi(z) + \frac{m}{2} z \phi'(z) + O(n^{-1} z^2 \phi''(z)).
    \end{equation*}
    Then, \eqref{eq. Plancherel-Rotach exp near R} follows after straightforward computations.

    Finally, assume that $\widetilde{z} \in {\rm III}$.
    Then, we write \eqref{eq. DKMVZ99 Plancherel-Rotach exp} as
    \begin{equation*}
        \psi_{n+m}(\sqrt{2n} z) = 2^{-\frac{3}{4}} \pi^{-\frac{1}{2}} (n+m)^{-\frac{1}{4}} e^{- (n+m) i \phi(\widetilde{z})} \bigg( \frac{(\widetilde{z}-1)^\frac{1}{4}}{(\widetilde{z}+1)^\frac{1}{4}} + \frac{(\widetilde{z}+1)^\frac{1}{4}}{(\widetilde{z}-1)^\frac{1}{4}} \bigg) \Big( 1 + O(n^{-1}) \Big).
    \end{equation*}
    Then, \eqref{eq. Plancherel-Rotach exp near R} follows from the power series expansion.
\end{proof}

\subsection{Asymptotics of the prekernel}
For the rest of this section, we denote
\begin{align}
\begin{split} \label{eq. def zeta, eta, z, w}
    \zeta = \sqrt{2n} z, \qquad \eta = \sqrt{2n} w,  \qquad 
    z = x + i \frac{y}{n}, \qquad  w = u + i \frac{v}{n}.
\end{split}
\end{align}
By combining the previous lemmas, we derive the asymptotic behaviour of the prekernel \eqref{eq. def of prekernel kappa}. 

\begin{lemma}[\textbf{Asymptotics of the prekernel on the diagonal near the real line}] \label{lem. asymp prekernel kappa(z,z)}
    Let $\zeta$ and $\eta$ be  given by \eqref{eq. def zeta, eta, z, w}. Then, as $n \to \infty$, the following asymptotic behaviours hold. 
    \begin{enumerate}
        \item Suppose $z \in {\rm I}_n$. Then, we have
        \begin{equation}
        \begin{split}\label{eq. asymp kappa diag I}
            \kappa_{n}(\zeta, \overline{\zeta}) &= \frac{i n}{ 16 \pi y^2 } \Big[ 4 y \sqrt{1-x^2} \cosh(4 y \sqrt{1-x^2}) - \sinh(4 y \sqrt{1-x^2} ) \Big] 
            \\
            &\quad + \cosh(4y\sqrt{1-x^2}) \Big( O(y^2) + O( (1-x^2)^{-2} y) \Big).
        \end{split}
        \end{equation}
        
        \item Suppose $z \in {\rm II}_n^+$ with $x = 1 + 2^{-1} n^{-\frac{2}{3}} \chi$. Then, we have
        \begin{equation} \label{eq. asymp kappa diag II}
            \kappa_{n}(\zeta,\overline{\zeta}) = n^\frac{2}{3} \Big[ \Airy( \chi ) \Big(O(\chi) + O(n^{-\frac{1}{3}}) \Big) + n^{-\frac{1}{3}} \Airy'( \chi ) \Big(O(\chi) + O(y) + O(n^{-\frac{1}{3}})\Big) \Big]^2.
        \end{equation}

        \item Suppose $z \in {\rm III}_n$. Then, we have
        \begin{equation} \label{eq. asymp kappa diag III}
            \kappa_{n}(\zeta,\overline{\zeta}) = n x^2 (x^2-1)^{-\frac{1}{2}} e^{- 2n i \phi(x) } O(1).
        \end{equation}
    \end{enumerate}
\end{lemma}
\begin{proof}
    Suppose $z\in{\rm I}_n$.
    It follows from \eqref{eq. Plancherel-Rotach osc near R} and elementary trigonometric identities that
    \begin{equation} \label{eq. psi tilde osc}
        \widetilde{\psi}_{n+m}(\sqrt{2n} z) =  2^{\frac{5}{4}} \pi^{-\frac{1}{2}} n^{\frac{3}{4}} (1-z^2)^{\frac{3}{4}} \Big[ - \cos(n\phi(z) - \theta_m(z)) + \cosh(n \im \phi(z)) O(n^{-1} (1-z^2)^{-\frac{5}{2}}) \Big],
    \end{equation}
    where the error term is uniform.
    Together with \eqref{eq. Plancherel-Rotach osc near R} and \eqref{eq. integral rep prekernel}, we have
    \begin{align}
    \begin{split} \label{eq. kappa(z,w) integral region I-I}
        \kappa_{n}(\zeta, \eta) &= \frac{n}{\pi} \int_0^1 t (1-w^2)^{-\frac{1}{4}} (1-s_t)^{\frac{3}{4}} 
        \\
        &\qquad \times\Big[ \cos(n\phi(s_t)-\theta_{-1}(s_t)) \cos(n\phi(w)-\theta_{0}(w))
         - \cos(n\phi(s_t)-\theta_{0}(s_t)) \cos(n\phi(w)-\theta_{-1}(w))  
        \\
        &\qquad\qquad + \cosh(n \im \phi(w)) \cosh(n \im \phi(s_t)) \Big(O(n^{-1} (1-w^2)^{-\frac{3}{2}} + O(n^{-1} (1-s_t^2)^{-\frac{5}{2}}\Big) \Big] \,dt
        \\
        & \quad + \frac{1}{2\pi} (1-z^2)^{-\frac{1}{4}} (1-w^2)^{-\frac{1}{4}} \Big[ \cos(n\phi(z)-\theta_{0}(z)) \cos(n\phi(w)-\theta_{-1}(w)) 
        \\
        &\qquad + \cosh(n \im \phi(z)) \cosh(n \im \phi(w)) \Big(O(n^{-1} (1-z^2)^{-\frac{3}{2}}) + O(n^{-1} (1-w^2)^{-\frac{3}{2}}\Big) \Big],
    \end{split}
    \end{align}
    where $s_t = t z + i (1-t) w$.
    For $\eta = \overline{\zeta}$, the leading asymptotic of the integral can be explicitly computed.
    More specifically, we use $z = x + i n^{-1} y$ to obtain uniform approximations
    \begin{align}
        (1-z^2)^{-a} &= (1-x^2)^{-a} \Big(1 + O(n^{-1}(1-x^2)^{-1}y) \Big), \label{eq. approx of (1-z^2)^a}
        \\
        \cos(n\phi(z) - \theta_m(z)) &= \cos\Big(n\phi(x)-\theta_m(x)+2iy\sqrt{1-x^2}\Big) + \cosh(2 y \sqrt{1-x^2}) O(n^{-1}y^2(1-x^2)^{-\frac{1}{2}}) \label{eq. approx of cos(n phi(z) - theta(z))}
    \end{align}
    for $a>0$, and analogous approximations for $s_t$ in place of $z$.
    Then, it follows from trigonometric identities that
    \begin{align*}
    \begin{split}
        \kappa_{n}(\zeta, \overline{\zeta}) &= \frac{i n}{\pi} \int_0^1 t (1-x^2) \sinh(4 t y \sqrt{1-x^2} ) \,dt + \cosh(4y\sqrt{1-x^2}) \Big[ O(y^2) + O((1-x^2)^{-2}y) \Big].
    \end{split}
    \end{align*}
    By evaluating the integral, we obtain \eqref{eq. asymp kappa diag I}.

    Secondly, we consider the case  $z \in {\rm II}_n^+$ with $x = 1 + 2^{-1}n^{-\frac{2}{3}} \chi$.
    By \eqref{eq. Plancherel-Rotach Airy near R} and Taylor expansion, we have
    \begin{equation}
        \widetilde{\psi}_{n+m}(\sqrt{2n} z) = n^{\frac{1}{4}} \Airy(\chi) \Big( O(\chi) + O(n^{-\frac{1}{3}}) \Big) + n^{-\frac{1}{12}} \Airy'(\chi) \Big( O(\chi) + O(y) + O(n^{-\frac{1}{3}}) \Big).
    \end{equation}
    Then, straightforward computations using \eqref{eq. integral rep prekernel} give rise to \eqref{eq. asymp kappa diag II}.

    Finally, let $z \in {\rm III}_n$ with $x > 1$.
    Then, it follows from \eqref{eq. Plancherel-Rotach exp near R}  that
    \begin{equation}
        \widetilde{\psi}_{n+m}(\sqrt{2n} z) = n^{\frac{3}{4}} z^2 e^{-(n+m)i\phi(z) + miz\phi'(z)/2 } \bigg( \frac{(z-1)^\frac{1}{4}}{(z+1)^\frac{1}{4}} + \frac{(z+1)^\frac{1}{4}}{(z-1)^\frac{1}{4}} \bigg) O(1).
    \end{equation}
    We uniformly approximate the exponent
    \begin{equation}
        -(n+m)i\phi(z) + \frac{miz}{2} \phi'(z) = - n i \phi(x) - 2 i y \sqrt{x^2-1} + m \log(x + \sqrt{x^2-1})  + O\Big( \frac{xy^2}{n \sqrt{x^2-1} } \Big).
    \end{equation}
    Substituting the above in \eqref{eq. integral rep prekernel}, we obtain
    \begin{equation}
        \kappa_{n}(\zeta,\overline{\zeta}) = n x e^{- 2n i \phi(x) } \bigg( \frac{(x-1)^\frac{1}{4}}{(x+1)^\frac{1}{4}} + \frac{(x+1)^\frac{1}{4}}{(x-1)^\frac{1}{4}} \bigg)^2 O(1).
    \end{equation}
    Then, by using \eqref{eq. def zeta, eta, z, w}, we have \eqref{eq. asymp kappa diag III}.
    The case $x<-1$ can be proved similarly, and we omit the details.
\end{proof}

\begin{lemma}[\textbf{Asymptotics of the prekernel on the off-diagonal near the real line}] \label{lem. asymp prekernel kappa(z,w)}
 Let $\zeta$ and $\eta$ be  given by \eqref{eq. def zeta, eta, z, w}. Then, as $n \to \infty$, the following asymptotic behaviour holds.  
    \begin{enumerate}
        \item Suppose $z, w \in {\rm I}_n$. Then,  we have
        \begin{align}
        \begin{split}\label{eq. kappa(z,w) region I-I O(n)}
            \kappa_{n}(\zeta,\eta) &= \cosh(4 y \sqrt{1-x^2}) \cosh(4 v \sqrt{1-u^2}) \Big( O(n) + O((1-x^2)^{-2}) + O((1-u^2)^{-2}) \Big) 
        \end{split}
        \end{align}
        and
        \begin{equation} \label{eq. kappa(z,w) region I-I O(1)}
            \kappa_{n}(\zeta,\eta) = (1-z^2)^{-\frac{1}{4}} (1-w^2)^{-\frac{1}{4}} \cosh(2 y \sqrt{1-x^2}) \cosh(2 v \sqrt{1-u^2})  \Big( O( (z-w)^{-1}) + O( n^{-1} (z-w)^{-2} ) \Big).
        \end{equation}
        \item Suppose $z \in {\rm I}_n$ and $w \in {\rm II}_n^+$ with $u = 1 + 2^{-1} n^{-\frac{2}{3}} \varsigma$.
        Then, we have
        \begin{equation} \label{eq. kappa(z,w) region I-II O(n^1/3-eps/4)}
            \kappa_{n}(\zeta,\eta) = n^{\frac{1}{3}} \cosh(2 y \sqrt{1 - x^2}) \Airy(\varsigma)  \Big( O((z-w)^{-1}) + O(n^{-1}(z-w)^{-2}) \Big).
        \end{equation}
        Furthermore, assume that $|z-w| \leq n^{(\epsilon - 2)/3}$ with $x = 1 + 2^{-1} n^{-\frac{2}{3}} \chi$. Then, we have
        \begin{equation} \label{eq. kappa(z,w) region I-II O(n^2/3)}
            \kappa_{n}(\zeta,\eta) = n^{\frac{2}{3}} \Airy(\chi) \Airy(\varsigma) O(1).
        \end{equation}

        \item Suppose $z, w \in {\rm II}_n^+$ with $x = 1 + 2^{-1} n^{-\frac{2}{3}} \chi$ and $u = 1 + 2^{-1} n^{-\frac{2}{3}} \varsigma$. Then, we have
        \begin{align}
        \begin{split} \label{eq. kappa(z,w) region II-II}
            \kappa_{n}(\zeta,\eta) &= n^{\frac{2}{3}} \Airy(\chi) \Airy(\varsigma) O(1).
        \end{split}
        \end{align}

        \item Suppose $z\in {\rm I}_n$ and $w\in {\rm III}_n$. Then, we have
        \begin{equation} \label{eq. kappa(z,w) region I-III}
            \kappa_{n}(\zeta,\eta) = \cosh(2 y \sqrt{1-x^2}) u^2 e^{- n i \phi(u)} O(n),
        \end{equation}

        \item Suppose $z\in {\rm II}_n^+ $ and $w\in {\rm III}_n$ with $x = 1 + 2^{-1} n^{-\frac{2}{3}} \chi$. Then, we have
        \begin{equation} \label{eq. kappa(z,w) region II-III}
            \kappa_{n}(\zeta,\eta) = \Airy(\chi) u^2 e^{- n i \phi(u)} O(n^2),
        \end{equation}

        \item Suppose $z, w \in {\rm III}_n$. Then, we have
        \begin{equation} \label{eq. kappa(z,w) region III-III}
            \kappa_{n}(\zeta,\eta) = |x u| (x^2-1)^{-\frac{1}{4}} (u^2-1)^{-\frac{1}{4}} e^{-n i \phi(x)} e^{-n i  \phi(u)} O(n).
        \end{equation}
    \end{enumerate}
\end{lemma}
\begin{proof}
    Suppose $z,w \in {\rm I}_n$. We may assume that $x^2 > u^2$, otherwise proceed with $\kappa_{n}(\eta, \zeta) = - \kappa_{n}(\zeta, \eta)$ instead.
    By using \eqref{eq. kappa(z,w) integral region I-I}, \eqref{eq. approx of (1-z^2)^a} and \eqref{eq. approx of cos(n phi(z) - theta(z))}, we have
    \begin{align*}
    \begin{split}
        \kappa_{n}(\zeta,\eta) &= \cosh(2v\sqrt{1-u^2}) \max_{t\in[0,1]} \Big\{ \cosh(2 \im s_{t}\sqrt{1-(\re s_{t})^2})\Big\} \Big( (1-u^2)^\frac{1}{2} O(n) + O((1-u^2)^{-2}) \Big)
        \\
        &\quad +  (1-z^2)^{-\frac{1}{4}} (1-w^2)^{-\frac{1}{4}} \cosh(2 y \sqrt{1-x^2}) \cosh(2 v \sqrt{1-u^2})  O( 1 ),
    \end{split}
    \end{align*}
    where $s_t = t z + (1-t)w$.
    By convexity, we have
    \begin{equation*}
        \max_{t\in[0,1]} \Big\{ \cosh(2 \im s_{t}\sqrt{1-(\re s_{t})^2})\Big\} \leq  \cosh(2 v \sqrt{1-u^2}) \cosh(2 y \sqrt{1-x^2}).
    \end{equation*}
    Therefore, we obtain
    \begin{equation*}
     \kappa_{n}(\zeta,\eta) = \cosh(4 v \sqrt{1-u^2}) \cosh(2 y \sqrt{1-x^2}) \Big( O(n) + O((1-u^2)^{-2}) \Big). 
    \end{equation*}
  To remove the assumption \( x^2 > u^2 \), we symmetrize the expression as in \eqref{eq. kappa(z,w) region I-I O(n)}.  
For \eqref{eq. kappa(z,w) region I-I O(1)}, we substitute \eqref{eq. Plancherel-Rotach osc near R} into \eqref{eq. rational rep prekernel}. The asymptotic formula \eqref{eq. kappa(z,w) region I-I O(1)} then follows from straightforward computations.
    
    Suppose $z \in {\rm I}_n$ and $w \in {\rm II}_n$. We use \eqref{eq. Plancherel-Rotach osc near R} and \eqref{eq. Plancherel-Rotach Airy near R} for \eqref{eq. rational rep prekernel}, and obtain
    \begin{equation*}
        \kappa_{n}(\zeta,\eta) = n^{\frac{1}{6}} (1-z^2)^{-\frac{1}{4}} \cosh(2y\sqrt{1-x^2}) \Airy(\varsigma) \Big( O((z-w)^{-1}) + O((z-w)^{-2}) \Big).
    \end{equation*}
    Since $z \in {\rm I}_n$, this implies \eqref{eq. kappa(z,w) region I-II O(n^1/3-eps/4)}.
    We further assume that $|z-w| \leq n^{(\epsilon - 2)/3}$.
    Then, we can also apply \eqref{eq. Plancherel-Rotach Airy near R} for the asymptotic of $\psi_{n+m}(\sqrt{2n} z)$.
    Thus from \eqref{eq. integral rep prekernel}, we have
    \begin{align}
    \begin{split} \label{eq. estimation kappa(z,w) II-II}
        \kappa_{n}(\zeta,\eta) &= n^{\frac{2}{3}} \Big( \Airy(2n^\frac{2}{3} (w-1)) O(1) + \Airy'(2n^\frac{2}{3} (w-1)) n^{-\frac{1}{3}} O(1) \Big)
        \\
        & \quad \times \max_{t\in[0,1]} \Big\{ \Airy(2n^\frac{2}{3}(s_t - 1)) \Big(O(s_t-1) + O(n^{-\frac{2}{3}}) \Big) 
        \\
        &\qquad \qquad \qquad + \Airy'(2n^\frac{2}{3}(s_t - 1)) n^{-\frac{1}{3}} \Big(O(s_t-1) + O(n^{-\frac{2}{3}}) \Big) \Big\}.
    \end{split}
    \end{align}
    Then, we obtain the desired asymptotic behaviour \eqref{eq. kappa(z,w) region I-II O(n^2/3)}.

   We prove the remaining cases in a similar manner.  
More specifically, we use the rational expression \eqref{eq. rational rep prekernel} when \( |z-w| > n^{(\epsilon-2)/3} \), and the integral representation \eqref{eq. integral rep prekernel} otherwise.  
Lemma~\ref{lem. P-R for the Hermite function} then provides the desired results.  
The proof for \eqref{eq. kappa(z,w) region II-II}–\eqref{eq. kappa(z,w) region III-III} is left to the interested reader.
\end{proof}

\begin{lemma}[\textbf{A bound of the prekernel in the whole domain}] \label{lem. asymp prekernel kappa(z,w) far away}
    Let $\zeta$ and $\eta$ be  given by \eqref{eq. def zeta, eta, z, w}.  Then,  there are constants $c=c(\alpha)$ and $d=d(\alpha)>0$ that satisfy
    \begin{equation} \label{eq. exp decay of kappa erfc}
        \kappa_{n}(\zeta,\eta) \sqrt{ \erfc\Big(\frac{2|y|}{\alpha}\Big) \erfc\Big(\frac{2 |v|}{\alpha}\Big) } = O\Big(e^{- c (y^2 + v^2)} e^{- d n ( |x| \sqrt{x^2-1}+ |u| \sqrt{u^2-1} )} \Big)
    \end{equation}
    for any $\zeta, \eta \in \C$.
\end{lemma}
\begin{proof}
    By \eqref{eq. rational rep prekernel} and \eqref{eq. integral rep prekernel}, it is sufficient to show that
    \begin{equation}
        \psi_{n+m}(\sqrt{2n} z) \sqrt{\erfc\Big( \frac{2 |y|}{\alpha} \Big)} = O(e^{- c y^2} e^{- d n |x| \sqrt{x^2-1}})
    \end{equation}
    for some constants $c,d>0$ and a fixed integer $m$.
    
    Suppose $z \in {\rm I}$.
    Since $\phi(z)$ and $\theta_m(z)$ are analytic in ${\rm I}$, we have 
    \begin{align}
    \begin{split} \label{eq. phi theta Taylor expansion}
        \Big| \phi(z) - \phi(x) - \phi'(x) i \frac{y}{n} \Big| &\leq \frac{y^2}{2 n^2} \max_{t \in [x, z]}\{|\phi''(t)|\},
        \\
        \Big| \theta_m(z) - \theta_m(x) \Big| &\leq \frac{y}{n} \max_{t \in [x, z]}\{|\theta_m'(t)|\}.
    \end{split}
    \end{align}
    Here, $[x,z]$ denotes a line segment in the complex plane connecting $x$ and $z$.
    Note that $\phi''(z) = -2z/\sqrt{1-z^2}$ and $\theta_m'(z) = -1/\sqrt{1-z^2}$. Thus, there exist some constants $c_1, c_2>0$ such that
    \begin{equation*}
        \im[n\phi(z) - \theta_m(z)] \leq 2 y \sqrt{1-x^2} + \frac{y^2}{2n} c_1 + \frac{|y|}{n} c_2
    \end{equation*}
    for any $z \in {\rm I}$.
 We compare the asymptotics of \(\psi_{n+m}(\sqrt{2n}z)\) and \(\erfc(2|y|/\alpha)\) using \eqref{eq. Plancherel-Rotach osc near R} and \eqref{eq. asymp of erfc}, and then conclude 
    \begin{equation*}
        \psi_{n+m}(\sqrt{2n} z)\sqrt{ \erfc(\frac{2 |y|}{\alpha})} = \exp\Big( -\frac{y^2}{\alpha^2} - n x \sqrt{x^2-1} \Big) O(1).
    \end{equation*}

    Next assume $z \in {\rm III}$.
    We further divide the domain into
    \begin{align*}
        {\rm III}^{(1)} = \{ z \in {\rm III}: |y| \leq \tfrac{n \delta}{2} |x| \},
        \qquad 
        {\rm III}^{(2)} = \{ z \in {\rm III}: |y| > \tfrac{n \delta}{2} |x| \},
    \end{align*}
    and analyse the real part of the exponent in \eqref{eq. DKMVZ99 Plancherel-Rotach exp}.
    By Taylor expansion of \eqref{eq. phi theta Taylor expansion}, there exists some constants $c_3, c_4>0$ such that
    \begin{align*}
        \re\Big[ -(n+m) i \phi(z) + \frac{m}{2} z i \phi'(z) \Big] \leq -(n+m) i \phi(x) + \frac{y^2}{2n} c_3 + \frac{|y|}{n} c_4
    \end{align*}
    for any $z \in {\rm III}^{(1)}$.
    In case of $z \in {\rm III}^{(2)}$, we have an elementary bound
    \begin{align*}
    \begin{split}
        \re\Big[ -(n+m) i \phi(z) + \frac{m}{2} z i \phi'(z)  \Big] &\leq n |z| |\sqrt{z^2-1}| + (n+m) \log(|z| + |\sqrt{z^2-1}|)
        \\
        &\leq \frac{2+\delta}{\delta} y \Big( \frac{2+\delta}{n\delta} y + 1 \Big)+ \frac{n+m}{n}\frac{4+2\delta}{\delta} y.
    \end{split}
    \end{align*}
    Since $z \in {\rm III}^{(2)}$, we have $|y| \geq n \delta \max\{1, \tfrac{|x|}{2}\}$.
    Comparing with \eqref{eq. asymp of erfc}, we obtain the desired asymptotic for both cases.
    
   Finally, suppose \( z \in \mathrm{II} \).  
Using \eqref{eq. asymp Airy(-z)} and \eqref{eq. asymp Airy(z)}, we obtain the desired asymptotic result through a similar estimation as shown in the cases \( z \in \mathrm{I} \) and \( \mathrm{III} \).  
This completes the proof.
\end{proof}

We are now ready to prove Theorem~\ref{Thm. LDP rate function} at weak non-Hermiticity. 

\begin{proof}[Proof of Theorem~\ref{Thm. LDP rate function} (ii)]
    We begin by expressing \eqref{eq. p(n,m) pfaffian integral representation} as
    \begin{equation*}
        p_{n,m} = \frac{p_{n,n}}{l!}\Big(\frac{2}{i}\Big)^{l} \int_{\HH^l} d^2\vzeta \, 
        \Pf\begin{bmatrix}
            \kappa_{n}(\zeta_j, \zeta_k) & \kappa_{n}(\zeta_j, \overline{\zeta}_k) 
            \smallskip 
            \\
            \kappa_{n}(\overline{\zeta}_j, \zeta_k) & \kappa_{n}(\overline{\zeta}_j, \overline{\zeta}_k)
        \end{bmatrix}_{j,k=1}^{l}
        \prod_{j=1}^{l} \erfc\Big( \frac{2 y_j}{\alpha} \Big),
    \end{equation*}
    using the notation \eqref{eq. def zeta, eta, z, w}.
    The Pfaffian admits an expansion
    \begin{equation*}
        \Pf \begin{bmatrix}
        \kappa_{n}(\zeta_j, \zeta_k) & \kappa_{n}(\zeta_j, \overline{\zeta}_k) \\
        \kappa_{n}(\overline{\zeta}_j, \zeta_k) & \kappa_{n}(\overline{\zeta}_j, \overline{\zeta}_k)
        \end{bmatrix}_{j,k=1}^l = \sum_{\sigma \in P_{2l}} \sgn(\pi_\sigma) \prod_{(j,k)\in \sigma} \kappa_{n}(\iota_j, \iota_k),
    \end{equation*}
    where $P_{2l}$ is a set of paired sequences given as
    \begin{align*}
        P_{2l} = \Big\{ (a_j, b_j)_{j=1}^l: \bigcup_{j=1}^l \{a_j, b_j\} = \{1, \ldots, 2l\}, \ a_1 < \cdots < a_l, \ a_j < b_j \text{ for all } j\Big\},
    \end{align*}
    and the permutation $\pi_\sigma$ associated to $\sigma = (a_j,b_j)_{j=1}^l \in P_{2l}$ is defined by
    \begin{align*}
        \pi_\sigma = \begin{pmatrix} 1 & 2 & \cdots & 2l-1 & 2l \\ a_1 & b_1 & \cdots & a_{2l} & b_{2l}\end{pmatrix},
    \end{align*}
    and $\iota_k$'s are defined by
    \begin{align*}
        &\iota_{2j-1} = \zeta_j, \quad \iota_{2j} = \overline{\zeta}_j, \qquad (j=1, \dots, l). 
    \end{align*} 
    Using the expansion, we have
    \begin{equation} \label{eq. proof of Thm1.1 wH ref 1}
        p_{n,m} = \frac{p_{n,n}}{l!} \Big( \frac{2}{i} \Big)^l \sum_{\sigma \in P_{2l}} \sgn(\pi_\sigma) \int_{\HH^l} d^2\vzeta \prod_{(a,b)\in \sigma} \kappa_{n}(\iota_a, \iota_b) \prod_{j=1}^l \erfc\Big(\frac{2 y_j}{\alpha}\Big).
    \end{equation}

  We first consider the summand in \eqref{eq. proof of Thm1.1 wH ref 1} with the index \(\sigma = (2j-1, 2j)_{j=1}^l\).  
In this case, the integral can be factorized, allowing us to utilise the result from the case \(l=1\) given by \eqref{eq. LDP wH l=1}.
    This yields
    \begin{align}
    \begin{split} \label{eq. integral of kappa factorized}
        \int_{\HH^l} d^2\vzeta \prod_{(a,b)\in \sigma} \kappa_{n}(\iota_a, \iota_b) \prod_{j=1}^l \erfc\Big(\frac{2 y_j}{\alpha}\Big) &= \prod_{j=1}^l \int_{\HH} d^2\zeta_j \, \kappa_{n}(\zeta_j, \overline{\zeta}_j) \erfc\Big(\frac{2 y_j}{\alpha}\Big) = \Big( \frac{i}{2} \frac{p_{n,n-2}}{p_{n,n}} \Big)^l.
    \end{split}
    \end{align} 
   Next, we turn to the cases with general $\sigma \in P_{2l}$.
 By using Lemmas~\ref{lem. asymp prekernel kappa(z,z)}, \ref{lem. asymp prekernel kappa(z,w)} and \ref{lem. asymp prekernel kappa(z,w) far away}, we have
    \begin{equation*}
        \int_{\HH^l} d^2\vzeta \prod_{(a,b)\in \sigma} \kappa_{n}(\iota_a, \iota_b) \prod_{j=1}^l \erfc\Big(\frac{2 y_j}{\alpha}\Big) = \int_{({\rm I_n} \cup {\rm II_n})^l} d^2\vzeta \prod_{(a,b)\in \sigma} \kappa_{n}(\iota_a, \iota_b) \prod_{j=1}^l \erfc\Big(\frac{2 y_j}{\alpha}\Big) \Big(1 + O(e^{-n^{\epsilon_i}})\Big).
    \end{equation*}
    In particular, by \eqref{eq. integral of kappa factorized}, this implies
    \begin{equation*}
        \int_{{\rm I}_n^l} d^2\vzeta \prod_{j=1}^l \cosh(2y_j\sqrt{1-x_j^2}) \prod_{j=1}^l \erfc\Big(\frac{2 y_j}{\alpha}\Big) =   \Big( \frac{i}{2} \frac{p_{n,n-2}}{p_{n,n}} \Big)^l (1 + O(n^{-\epsilon})) = C n^l (1 + O(n^{-\epsilon}))
    \end{equation*}
    for some constants $C>0$ and $\epsilon \in (0,1)$.
    Furthermore, for any $\vzeta \in ({\rm I_n} \cup {\rm II_n})^l$, we have
    \begin{align*}
        \prod_{(a,b)\in \sigma} \kappa_{n}(\iota_a, \iota_b) = \bigg( \prod_{j=1}^l \cosh(4y_j\sqrt{1-x_j^2}) \, \bigg) O(n^l).
    \end{align*}
    With the notation \eqref{eq. def zeta, eta, z, w}, let us define a set
    \begin{align*}
        \Upsilon_n = \{ \vzeta \in ({\rm I}_n \cup {\rm II}_n)^l : |x_j - x_k| > n^{\epsilon_r-1} \text{ for } 1 \leq j \neq k \leq l \}.
    \end{align*}
    If $\vzeta \in \Upsilon_n$ and $\sigma \neq (2j-1, 2j)_{j=1}^l$, then we have
    \begin{align*}
        \prod_{(a,b)\in \sigma} \kappa_{n}(\iota_a, \iota_b) = \bigg( \prod_{j=1}^l \cosh(2y_j\sqrt{1-x_j^2}) \, \bigg) O(n^{l - \epsilon'})
    \end{align*}
    for some constant $\epsilon'>0$.
    Combining the above, for $\sigma \neq (2j-1, 2j)_{j=1}^l$, we have 
    \begin{align*}
        \int_{({\rm I_n} \cup {\rm II_n})^l} d^2\vzeta \prod_{(a,b)\in \sigma} \kappa_{n}(\iota_a, \iota_b) \prod_{j=1}^l \erfc\Big(\frac{2 y_j}{\alpha}\Big) &= \bigg( \int_{\Upsilon_n} + \int_{({\rm I_n} \cup {\rm II_n})^l \setminus \Upsilon_n} \bigg) d^2\vzeta \prod_{(a,b)\in \sigma} \kappa_{n}(\iota_a, \iota_b) \prod_{j=1}^l \erfc\Big(\frac{2 y_j}{\alpha}\Big)
        \\
        &= O(n^{l-\epsilon'}) + O(n^{l+\epsilon_r-1}).
    \end{align*}
    Combining all of the above, we obtain 
    \begin{align*}
        p_{n,m} &= \frac{p_{n,n}}{l!} \Big( \frac{2}{i} \Big)^l \prod_{j=1}^l \int_{\HH} d^2\zeta_j \, \kappa_{n}(\zeta_j, \overline{\zeta}_j) \erfc\Big(\frac{2 y_j}{\alpha}\Big) \Big( 1 + O(n^{-\epsilon''}) \Big) =  \frac{p_{n,n}}{l!} \Big(\frac{p_{n,n-2}}{p_{n,n}} \Big)^l \Big( 1 + O(n^{-\epsilon''}) \Big)
    \end{align*}
    for some constant $\epsilon''>0$. This completes the proof. 
\end{proof}

We provide a heuristic explanation for why the pairing $\sigma = (2j - 1, 2j)_{j = 1}^l$ contributes to the leading asymptotics in \eqref{eq. proof of Thm1.1 wH ref 1}.  
As in the strong non-Hermiticity regime, one can compute the approximated potential $Q_n^{(r)}(z)$ defined in \eqref{eq. def of Q_n^(r)}.  
However, since $\tau = 1 - \alpha^2 / n$, the difference $Q_n^{(r)}(z_n^\star) - Q_n^{(r)}(0)$ in \eqref{eq. def of z star} vanishes as $n \to \infty$.  
Instead, $Q_n^{(r)}(z)$ forms a potential well in the upper half-plane near the real axis, see Figure~\ref{fig. 3d plot potential} for an illustration.  
As a result, the complex eigenvalues are typically distributed within this potential well, which forms an elongated neighbourhood along $[-2,2] \subset \mathbb{C}$.  
This heuristically suggests that the typical distance between complex eigenvalues in the upper half-plane is of order $1$, implying that the interaction between distinct complex-conjugate pairs may be neglected.  
In terms of the Pfaffian in \eqref{eq. proof of Thm1.1 wH ref 1}, this suggests that the block-diagonal contribution corresponding to $\sigma = (2j - 1, 2j)_{j = 1}^l$ dominates in the computation.

\subsection*{Acknowledgements} The authors are grateful to the DFG-NRF International Research Training Group IRTG 2235 supporting the Bielefeld-Seoul graduate exchange programme (NRF-2016K2A9A2A13003815). 
Sung-Soo Byun was supported by the New Faculty Startup Fund at Seoul National University and by the National Research Foundation of Korea grant (RS-2023-00301976, RS-2025-00516909).
The authors are very grateful to the anonymous referee for useful comments and in particular for an improvement of the proof of Lemma~\ref{lem. asymp vandermonde integral}.

\end{document}